\documentclass{article}
\usepackage{amsfonts,amsmath,amssymb,amsthm,amscd,latexsym,geometry}
\usepackage{graphicx}
\usepackage{color}

\newcommand{\R}{\mathbb{R}}
\newcommand{\PP}{\mathbb{P}}
\newcommand{\EE}{\mathbb{E}}
\newcommand{\Ba}{\mathcal{B}}
\newcommand{\e}{\varepsilon}
\newcommand{\D}{\Delta}
\newcommand{\ZZ}{\mathcal{N}}
\newcommand{\RR}{\mathcal{R}}

\newtheorem{theorem}{Theorem}
\newtheorem{criterion}{Criterion}

\newtheorem{proposition}{Proposition}

\title{Multiscale Piecewise Deterministic Markov Process in Infinite Dimension: Central Limit Theorem and Langevin Approximation} 

\author{
   A. Genadot\thanks{Laboratoire de Probabilit\'es et Mod\`eles Al\'eatoires, Universit\'e Pierre et Marie Curie, Paris 6, Case courrier 188, 4 Place Jussieu, 75252 Paris Cedex 05, France. This work has been supported by the Agence Nationale de la Recherche through the project MANDy, Mathematical Analysis of Neuronal Dynamics, ANR-09-BLAN-0008-01.} \thanks{email: algenadot@gmail.com}
   \and
   M. Thieullen$^*$\thanks{email: michele.thieullen@upmc.fr}
}

\date{Version 1}

\begin{document}



\maketitle
\begin{abstract}
In \cite{GT}, the authors addressed the question of the averaging of a slow-fast Piecewise Deterministic Markov Process (PDMP) in infinite dimension. In the present paper, we carry on and complete this work by the mathematical analysis of the fluctuation of the slow-fast system around the averaged limit. A central limit theorem is derived and the associated Langevin approximation is considered. The motivation of this work is a stochastic Hodgkin-Huxley model which describes the propagation of an action potential along the nerve fiber. We study this PDMP in detail and provide more general results for a class of Hilbert space valued PDMP.
\end{abstract}
\section{Introduction}

In \cite{GT}, the authors addressed the question of the averaging for the multiscale stochastic Hodgkin-Huxley model. This model describes the evolution of an action potential or nerve impulse along the nerve axon of a neuron with a finite number of channels which display stochastic gating mechanisms. Mathematically, this stochastic Hodgkin-Huxley model belongs to the class of Piecewise Deterministic Markov Processes (PDMP) with multiple time scales. In \cite{GT} we derived averaging results for this class of models. The averaged model is still a PDMP of lower dimension. In the present paper, we study the fluctuation of the original slow-fast system around its averaged limit. A central limit theorem is derived and the associated Langevin approximation is considered. A numerical example is also provided at the end of the paper.\\
The mathematical analysis of PDMP, and more generally of hybrid systems, constitutes a very active area of current research since a few years. An hybrid system can be defined as a dynamical system describing the interactions between a continuous macroscopic dynamic and a discrete microscopic one. For PDMPs, between the jumps, the motion of the macroscopic component is given by a deterministic flow. It is essentially this fact that gives the PDMPs their most important peculiarity as hybrid systems: they enjoy the Markov property.\\
The PDMPs, also called markovian hybrid systems, have been introduced by Davis in \cite{Davis1,Davis2} for the finite dimensional setting and generalized in \cite{BR} to cover the infinite dimensional case. Recently, the asymptotic behavior of PDMPs have been investigated in \cite{BLMZ1,BLMZ2,CDMR,T}, limit theorems for infinite dimensional PDMPs in \cite{RTW}, control problems in \cite{CD_C1,CD_C2,Goreac}, numerical methods in \cite{R,Sap}, time reversal in \cite{LP} and to end up this list with no claim of completeness, estimation of the jump rates for PDMPs in \cite{Azais}.\\ 
Hybrid systems are the object of great attention because they offer an accurate description of a large class of phenomena arising in various domains such as physics or biology. For example, in mathematical neuroscience, a domain the authors are more particularly interested in, PDMP models arise naturally in the description of the propagation of the nerve impulse, see \cite{Austin, BR}. This mathematical description has been proved to be consistent with classical models such as the Hodgkin-Huxley model and compartment type models, see \cite{RTW}.\\
In this paper, the authors are interested in the question of averaging for a spatially extended PDMP model of propagation of the nerve impulse and corresponding fluctuations. Averaging is of first importance because it allows to simplify the dynamic of a system which contains intrinsically two different time-scales. Moreover, the averaged limit preserves the qualitative behavior of the original dynamical system, see \cite{PS,YZ}. In the finite dimensional case these questions have been addressed in \cite{Wainrib2} and \cite{Fa_Cri}.\\
Let us mention at this point that the study of slow-fast systems of Stochastic Partial Differential Equations (SPDEs), a framework  different from ours but very instructive, is an area of very active research. Averaging results have been derived in \cite{CF1, RW} and fluctuation around the limit and large deviations have been studied in \cite{DRW}.\\
The paper is organized as follows. Section \ref{Not} gathers the main notations in use throughout the text. In Section \ref{Mod} and \ref{Sing} we recall as briefly as possible the model and the main results of \cite{GT} and in particular the different properties of the averaged process. Section \ref{Main} introduces the main results of the present paper: the central limit theorem and the attached Langevin approximation are stated. The description of the general class of PDMP which can be included in our framework is described. (Section \ref{GF}). In Section \ref{Preuve}, we begin by proving the Central Limit Theorem in the so-called all fast case before considering the multi-scale case, this is Section \ref{All} and \ref{Multi}. We divide the proof in the all fast case in two parts: the tightness in Section \ref{Tight} and the identification of the limit in Section \ref{Id}. Properties of the diffusion operator related to the fluctuations are investigated in Section \ref{Diff_op}. In Section \ref{Langevin}, the Langevin approximation associated to the averaged model and its fluctuations is considered. A numerical example is presented in Section \ref{Exemple}.
\subsection{Notations}\label{Not}

Let $I=[0,1]$. $L^2(I)$ is the space of measurable and squared integrable functions. It is a Hilbert space endowed with the usual scalar product
\[
(f,g)_{L^2(I)}=\int_I f(x)g(x)dx
\]
and norm $\|\cdot\|^2_{L^2(I)}=(\cdot,\cdot)_{L^2(I)}$. $H=H^1_0(I)$ denotes the completion of the set of $\mathcal{C}^\infty$ functions with compact support on $I$ with respect to the norm $\|\cdot\|_H$ defined by
\[
\|f\|_H=\sqrt{\int_I (f(x))^2+(f'(x))^2dx}
\]
H is also a Hilbert space and we denote its scalar product simply by $(\cdot,\cdot)$. A Hilbert basis of $H$ (resp. $L^2(I)$) is given by the following functions on $I$
\[
e_k(\cdot)=\frac{\sqrt2}{\sqrt{1+(k\pi)^2}}\sin(k\pi \cdot)\quad (\text{resp. } f_k(\cdot)=\sqrt2\sin(k\pi \cdot))
\]
for $k\geq1$. The dual space of $H$ which is $H^{-1}$ is denoted by $H^*$. $<\cdot,\cdot>$ is the duality pairing between $H$ and $H^*$. The triple of Banach spaces $H\subset L^2(I)\subset H^*$ is an evolution triple or Gelfand triple. The embeddings in between these three spaces are continuous and dense.  For any $h\in L^2(I)$ and any $u\in H$: $ <h,u>=(h,u)_{L^2(I)}$ and for any $x\in I$, $k\geq1$
\[
<\delta_x,e_k>=(1+(k\pi)^2)e_k(x)
\]
The embedding $H\subset \mathcal{C}(I,\R)$ also holds and we denote by $C_P$ the constant such that, for all $u\in H$
\[
\sup_{I}|u|\leq C_P\|u\|_H
\]
We refer the reader to \cite{Henry}, Chapter 1, Section 1.3 for more details.\\ 
In $H$, the Laplacian with zero Dirichlet boundary conditions has the following spectral decomposition
\begin{equation*}
\D u=-\sum_{k\geq1}(k\pi)^2(u,e_k)e_k
\end{equation*}
for $u$ in the domain $\mathcal{D}(\D)=\{u\in H;\sum_{k\geq1}k^4 (u,e_k)^2<\infty\}$. It generates the semi-group of operators $\{e^{\D t},t\geq0\}$ defined for $u\in H$ by
\[
e^{\D t}u=\sum_{k\geq1}e^{-(k\pi)^2t}(u,e_k)e_k
\]
We say that a function $f:H\mapsto\R$ has a Fr\'echet derivative in $u\in H$ if there exists a bounded linear operator $T_u:H\mapsto \R$ such that
\[
\lim_{h\to 0}\frac{f(u+h)-f(u)-T_u(h)}{\|h\|_H}=0
\]
We then write $\frac{df}{du}(u)$ for the operator $T_u$. For example, the square of the $\|\cdot\|_H$-norm and the Dirac distribution in $x\in I$ are Fr\'echet differentiable on $H$. For all $u\in H$
\[
\frac{d\|\cdot\|^2_H}{du}(u)[h]=2(u,h),\quad \frac{d\delta_x}{du}(u)[h]=h(x)
\]
for all $h\in H$. In the same way, we can define the Fr\'echet derivative of order $2$. The second Fr\'echet derivative of a twice Fr\'echet differentiable function $f:H\to\R$ is denoted by $\frac{d^2f}{du^2}(u)$. It can be considered as a bilinear form on $H\times H$. For instance
\[
\frac{d^2\|\cdot\|^2_H}{du^2}(u)[h,k]=2(h,k),\quad \frac{d^2\delta_x}{du^2}(u)[h,k]=0,
\]
for all $(h,k)\in H\times H$. Fr\'echet differentiation is stable by summation and multiplication.

\subsection{The model}\label{Mod}

In this section, we introduce the multiscale stochastic Hodgkin-Huxley model. This model was first considered in \cite{Austin}, and later in \cite{BR,GT, RTW}. Although we are interested in the multi scale stochastic Hodgkin-Huxley model, we start by describing the model that does not display different time scales, for the sake of clarity. The spatially extended  stochastic Hodgkin-Huxley model describes the propagation of an action potential along an axon at the scale of the ion channels. The axon, or nerve fiber, is the component of a neuron which allows the propagation of an incoming signal from the soma to another neuron on long distances. The length of the axon is large relative to its radius, thus, for mathematical convenience, we consider the axon as a segment $I$. We choose here $I=[0,1]$. All along the axon are the ion channels which amplify and allow the propagation of the incoming impulse. We assume that there are $N\geq1$ ion channels along the axon at the locati!
 ons $z_i\in I$ for $i\in\ZZ\subset]0,1[$ with $|\ZZ|=N$. $\ZZ$ is thus a finite set. In \cite{Austin,GT} for instance, $\ZZ=\mathbb{N}\cap{N\mathring{I}}$, which means that the ion channels are regularly spaced. Each ion channel can be in a state $\xi\in E$ where $E$ is a finite state space, for instance, in the Hodgkin-Huxley model, a state can be: "receptive to sodium ions and open". When a ion channel is open, it allows some ionic species to enter or leave the cell, generating in this way a current. For a greater insight into the underlying biological phenomena governing the model, the authors refer to \cite{Hille}.\\
The ion channels switch between states according to a continuous time Markov chain whose jump intensities depend on the local potential of the axon membrane. For two states $\xi,\zeta\in E$ we define by $\alpha_{\xi,\zeta}$ the jump intensity or transition rate function from the state $\xi$ to the state $\zeta$. It is a real valued function of a real variable supposed to be, as its derivative, Lipschitz-continuous. We assume moreover that: $0\leq\alpha_{\xi,\zeta}\leq\alpha^+$  for any $\xi,\zeta\in E$ and either $\alpha_{\xi,\zeta}$ is constant equal to zero or is strictly positive bounded below by a strictly positive constant $\alpha_-$. That is, the non-zero rate functions are bounded below and above by strictly positive constants. For a given channel, the rate function describes the rate at which it switches from one state to another.\\
A possible configuration of all the $N$ ion channels is denoted by $r=(r(i),i\in\ZZ)$, a point in the space of all configurations $\RR=E^{\ZZ}$: $r(i)$ is the state of the channel located at $z_i$, for $i\in\ZZ$. The channels, or stochastic processes $r(i)$, are supposed to evolve independently over infinitesimal time-scales. Denoting by $u_t(z_i)$ the local potential at point $z_i$ at time $t$, we have
\begin{equation}\label{saut}
\PP(r_{t+h}(i)=\zeta|r_t(i)=\xi)=\alpha_{\xi,\zeta}\left(u_t\left(z_i\right)\right)h+o(h)
\end{equation}
For any $\xi\in E$ we also define the maximal conductance $c_\xi$ and the steady state potentials, or driven potentials, $v_\xi$ of a channel in state $\xi$ which are both  constants, the first being non negative.\\
The transmembrane potential $u_t(x)$, that is the difference of electrical potential between the outside and the inside of the axon, evolves according to the following hybrid reaction-diffusion PDE
\begin{equation}\label{dyn}
\partial_t u_t=\nu\Delta u_t +\frac{1}{N}\sum_{i\in \ZZ}c_{r_t(i)}(v_{r_t(i)}-u_t(z_i))\delta_{z_i}
\end{equation}
We assume zero Dirichlet boundary conditions for this PDE (clamped axon). The positive constant $\nu$ is the intensity of the diffusion part of the above PDE, for simplicity in the notations we assume that $\nu=1$. We are interested in the process $(u_t,r_t)_{t\in[0,T]}$. The applied current given by the reaction term of the PDE is denoted in the sequel by $G_r(u)$ and and satisfies
\begin{equation}
G_r(u)=\frac{1}{N}\sum_{i\in \ZZ}c_{r(i)}(v_{r(i)}-u(z_i))\delta_{z_i}\in H^*
\end{equation}
for $(u,r)\in H\times\RR$. The following result of \cite{Austin} states that there exists a stochastic process satisfying equations (\ref{dyn}) and (\ref{saut}). Let $u_0$ be in $H$ such that $\min_{\xi\in E}v_\xi\leq u_0\leq\max_{\xi\in E}v_\xi$, the initial potential of the axon. Let $q_0\in\RR$ be the initial configuration of the ion channels.
\begin{proposition}[\cite{Austin},\cite{BR}]\label{prop_A1}
Fix $N\geq 1$ and let $(\Omega, \mathcal{F},(\mathcal{F}_t)_{0\leq t\leq T},\PP)$ be a filtered probability space satisfying the usual conditions. There exists a pair $(u_t,r_t)_{0\leq t\leq T}$ of c\`adl\`ag adapted stochastic processes such that each sample path of $u$ is in $\mathcal{C}([0,T],H)$ and $r_t$ is in $\RR$ for all $t\in[0,T]$ and $(u_t,r_t)_{0\leq t\leq T}$ satisfies (\ref{dyn}-\ref{saut}). Moreover  $(u_t,r_t)_{0\leq t\leq T}$ is a so called Piecewise Deterministic Markov Process.
\end{proposition}
\noindent The existence of a stochastic process solution of (\ref{dyn}) and (\ref{saut}) has been first proved in \cite{Austin}. In this paper, the author uses the Schaeffer fixed point theorem to show that when the jump process $r$ jumps at rate $1$, there exists a solution to (\ref{dyn}) and uses the Girsanov theorem for c\'adl\'ag processes with finite state space to recover the dynamic of $r$. Another approach has been developed in \cite{BR}. In this paper, the authors construct explicitly the process $(u,r)$ as a piecewise deterministic Markov process generalizing in this way the theory of piecewise deterministic Markov process developed by Davis from the finite to the infinite dimension setting, see \cite{Davis1,Davis2}. The authors in \cite{BR} prove that their process is markovian and moreover characterize its generator. Another approach based on the marked point process theory is also possible, see for instance \cite{J} and the extension to our framework in \cite{RTW}.\\
We proceed now by recalling the form of the generator of the process $(u,r)$. For $(u_0,r)\in H\times \RR$, we denote by $(\psi_{r}(t,u_0),t\in[0,T])$ the unique solution starting from $u_0$ of the PDE
\begin{equation}\label{dyn:fix}
\partial_t u_t=\Delta u_t +\frac{1}{N}\sum_{i\in \ZZ}c_{r(i)}(v_{r(i)}-u_t(z_i))\delta_{z_i}
\end{equation}
with zero Dirichlet boundary conditions. 
\begin{proposition}
Let $f$ be a locally bounded measurable function on $H\times\RR$ such that the map $t\mapsto f(\psi_r(t,u_0),r)$ is continuous for all $(u_0,r)\in H\times\RR$. Then $f$ is in the domain $\mathcal{D}(\mathcal{A})$ of the extended generator of the process $(u,r)$. The extended generator is given for almost all $t$ by
\begin{equation}\label{gene}
\mathcal{A}f(u_t,r_t)=\frac{d}{dt}f(u_\cdot,r_t)(t)+\Ba(u_t) f(u_t,\cdot)(r_t)
\end{equation}
where
\[
\Ba(u_t) f(u_t,\cdot)(r_t)=\sum_{i\in\ZZ}\sum_{\zeta\in E}[f(u_t,r_t(r_t(i)\to\zeta))-f(u_t,r_t)]\alpha_{r_t(i),\zeta}(u_t(z_i))
\]
with $r_t(r_t(i)\to\zeta)$ is the element of $\RR$ with $r_t(r_t(i)\to\zeta)(j)$ equal to $r_t(j)$ if $j\neq i$ and to $\zeta$ if $j=i$. The notation $\frac{d}{dt}f(u_\cdot,r_t)(t)$ means that the function $s\mapsto f(u'_s,r)$ is differentiated at $s=t$, where $u'$ is the solution of the PDE (\ref{dyn:fix}) with the channel state $r_t$ held fixed equal to $r$.
\end{proposition}

\medskip

\noindent Let us mention at this point that in Section \ref{Langevin} we will work with a slightly different model where the Dirac distributions $\delta_{z_i}$ are replaced by approximations $\phi_{z_i}$ in the sense of distributions, in the same way as in so called compartment models (see Section \ref{Main} for more details). In this model the reaction term is given by
\[
G_{r,\phi}(u)=\frac{1}{N}\sum_{i\in \ZZ}c_{r(i)}(v_{r(i)}-(u,\phi_{z_i})_{L^2(I)})\phi_{z_i}
\]
for $(r,u)\in\RR\times L^2(I)$. For $i\in\ZZ$, the function $\phi_{z_i}$ which belongs to $L^2(I)$ approximates the Dirac distribution $\delta_{z_i}$. Replacing $\delta_{z_i}$ by $\phi_{z_i}$ corresponds to consider that when the channel located at $z_i$ is open and allows a current to pass, not only the voltage at the point $z_i$ is affected, but also the voltage on a small area around $z_i$ (see \cite{BR}). The family of functions $\phi_{z_i}$  is indexed by a parameter $\kappa$ related to the membrane area considered: the smaller $\kappa$, the smaller the area. When $u$ is held fixed, the dynamic of the ion channel at location $z_i$ is given by
\begin{equation}\label{saut:phi}
\PP(r_{t+h}(i)=\zeta|r_t(i)=\xi)=\alpha_{\xi,\zeta}\left((u,\phi_{z_i})_{L^2(I)}\right)h+o(h)
\end{equation}
for $\xi,\zeta\in E$ and $t,h\geq0$. 

\noindent All the results stated for the model with the Dirac mass $\delta_{z_i}$ hold also true for the model with the approximations $\phi_{z_i}$. We prefer to work in a first part with the model introduced above (that is to say with the Dirac distribution) because it corresponds exactly to the model studied in \cite{Austin, GT}. However, when considering the Langevin approximation associated to the central limit theorem, the formulation with the mollifier $\phi_{z_i}$ appears to be more tractable.

\subsection{Singular perturbation and previous averaging results}\label{Sing}

In this section, we introduce a slow-fast dynamic in the stochastic Hodgkin-Huxley model: some states of the ion channels communicate faster between each other than others. This is biologically relevant as remarked for example in \cite{Hille}. Mathematically, this leads to the introduction of an additional small parameter $\e>0$ in our previously described model: the states which communicate at a faster rate communicate at the previous rate $\alpha_{\xi,\zeta}$ divided by $\e$. For an introduction on slow-fast system, we refer to \cite{PS}, for a general theory of slow-fast continuous time Markov chain, see \cite{YZ} and for the case of slow-fast systems with diffusion, see \cite{BG}.\\
We make a partition of the state space $E$ according to the different orders in $\e$ of the rate functions
\[
E=E_1\sqcup\cdots\sqcup E_l
\]
where $l\in\{1,2,\cdots\}$ is the number of classes. Inside a class $E_j$, the states communicate faster at jump rates of order $\frac1\e$. States of different classes communicate at the usual rate of order $1$. For $\e>0$ fixed, we denote by $(u^\e,r^\e)$ the modification of the PDMP introduced in the previous section with now two time scales. Its generator is, for $f\in\mathcal{D}(\mathcal{A^\e})$
\begin{equation}\label{gene_eps}
\mathcal{A^\e}f(u^\e_t,r^\e_t)=\frac{d}{dt}f(u^\e_\cdot,r^\e_t)(t)+\Ba^\e (u^\e_t)f(u^\e_t,\cdot)(r^\e_t)
\end{equation}
$\Ba^\e$ is the component of the generator related to the continuous time Markov chain $r^\e$. According to (\ref{gene}) and our slow-fast description we have the two time scales decomposition of this generator
\begin{equation*}\label{gene_saut}
\Ba^\e =\frac1\e\Ba+\hat\Ba
\end{equation*}
where the "fast" generator $\Ba$ is given by
\begin{equation*}\label{gene_fast}
\Ba(u^\e_t) f(u^\e_t,r^\e_t)=\sum_{i\in\ZZ}\sum_{j=1}^l1_{E_j}(r^\e_t(i))\sum_{\zeta\in E_j}[f(u^\e_t,r^\e_t(r^\e_t(i)\to\zeta))-f(u^\e_t,r^\e_t)]\alpha_{r^\e_t(i),\zeta}(u^\e_t(z_i))
\end{equation*}
and the "slow" generator $\hat\Ba$ is given by
\begin{equation*}\label{gene_slow}
\hat\Ba(u^\e_t) f(u^\e_t,r^\e_t)=\sum_{i\in\ZZ}\sum_{j=1}^l1_{E_j}(r^\e_t(i))\sum_{\zeta\notin E_j}[f(u^\e_t,r^\e_t(r^\e_t(i)\to\zeta))-f(u^\e_t,r^\e_t)]\alpha_{r^\e_t(i),\zeta}(u^\e_t(z_i))
\end{equation*}
For $y\in \R$ fixed and $g:\R\times E\to\R$, we denote by $\Ba_j(y)$, $j\in\{1,\cdots,l\}$ the following generator
\[
\Ba_j(y)g(\xi)=1_{E_j}(\xi)\sum_{\zeta\in E_j}[g(y,\zeta)-g(y,\xi)]\alpha_{\xi,\zeta}(y)
\]
For any $y\in \R$ fixed, and any $j\in\{1,\cdots,l\}$, we assume that the fast generator $\Ba_j(y)$ is weakly irreducible on $E_j$, i.e has a unique quasi-stationary distribution denoted by $\mu_j(y)$.  This quasi-stationary distribution is supposed to be Lipschitz-continuous in $y$, as well as its derivative.\\
Following \cite{YZ}, the states in $E_j$ can be considered as equivalent. For any $i\in\ZZ$ we define the new stochastic process $(\bar r_t^\epsilon)_{t\geq0}$ by $\bar r_t^\e(i)=j$ when $ r_t^\e(i)\in E_j$ and abbreviate $E_j$ by $j$ . We then obtain an aggregate process $\bar r^\e(i)$ with values in $\{1,\cdots,l\}$. This process is also often called the coarse-grained process. It is not a Markov process for $\e>0$ but a Markovian structure is recovered at the limit when $\e$ goes to $0$. More precisely, we have the following proposition
\begin{proposition}[\cite{YZ}]\label{YZ}
For any $y\in\R$, $i\in\ZZ$, the process $\bar r^\e(i)$ converges weakly when $\e$ goes to $0$ to a Markov process $\bar r(i)$ generated by
\[
\bar\Ba(y) g(\bar r(i))=\sum_{j=1}^l 1_{j}(\bar r(i))\sum_{k=1,k\neq j}^l (g(k)-g(j))\sum_{\xi\in E_j}\sum_{\zeta\in E_k}\alpha_{\zeta,\xi}(y)\mu_j(y)(\zeta)
\]
with $g:\{1,\cdots,l\}\to\R$ measurable and bounded.
\end{proposition}
\noindent In order to determine the complete limit when $\e$ goes to zero, we need to average the reaction term $G_r(u)$ against the quasi-invariant distributions. That is we consider that each cluster of states $E_j$ has reached its stationary behavior. For any $\bar r\in\bar\RR=\{1,\cdots,l\}^{\ZZ}$ we define the averaged function by
\begin{equation}\label{F}
F_{\bar r}(u)=\frac{1}{N}\sum_{i\in \ZZ}\sum_{j=1}^l 1_{j}(\bar r(i))\sum_{\zeta\in E_j}c_\zeta \mu_j(u(z_i))(\zeta)(v_\zeta-u(z_i))\delta_{z_i}
\end{equation}
Therefore, we call the following PDE 
\begin{equation}\label{Au_moy}
\partial_t u_t=\D u_t + F_{\bar r_t}(u_t)
\end{equation}
the {\it averaged equation} of (\ref{dyn}). We take zero Dirichlet boundary conditions and initial conditions $u_0$ and $\bar q_0$ where $\bar q_0$ is the aggregation of the initial channel configuration $q_0$. In equation (\ref{Au_moy}), the jump process $(\bar r_t)_{t\in[0,T]}$ evolves, each coordinate independently over infinitesimal time intervals, according to the averaged jump rates between the subsets $E_j$ of $E$. For $k$ and $j$ in $\{1,\cdots,l\}$, the average jump rate from class $E_j$ to class $E_k$ is given by
\begin{equation}
\bar \alpha_{jk}(y)=\sum_{\zeta\in E_j}\sum_{\xi\in E_k}\alpha_{\zeta,\xi}(y)\mu_j(y)(\zeta)
\end{equation}
We recall now the most important results of \cite{GT} and we refer the interested reader to this paper for proofs. The first important result is the uniform boundedness in $\e$ of the process $u^\e$.
\begin{proposition}[\cite{GT}]\label{Prop_bound}
For any $T>0$, there is a deterministic constant $C>0$ independent of $\e\in]0,1]$ such that
\begin{equation*}
\sup_{t\in[0,T]} \|u^\e_t\|_H\leq C
\end{equation*}
almost surely.
\end{proposition}
\noindent The second result states that the averaged model is well posed and is still a PDMP.
\begin{proposition}[\cite{GT}]\label{APDMP}
For any $T>0$ there exists a probability space satisfying the usual conditions such that equation (\ref{Au_moy}) defines a PDMP $(u_t,\bar{r}_t)_{t\in[0,T]}$ in infinite dimension in the sense of \cite{BR}. Moreover, there is a constant $C$ such that
\begin{equation*}
\sup_{t\in[0,T]}\|u_t\|_H\leq C
\end{equation*}
and $u\in\mathcal{C}([0,T],H)$ almost-surely.
\end{proposition}
\noindent We can now state the averaging result.
\begin{theorem}[\cite{GT}]\label{APDE}
When $\e$ goes to $0$ the stochastic process $u^\e$ solution of (\ref{gene_eps}) converges in distribution in the space $\mathcal{C}([0,T],H)$ to $u$, solution of (\ref{Au_moy}).
\end{theorem}

\section{Main results}\label{Main}

\subsection{Fluctuations and Langevin Approximation}

We present in this section the main results of the present paper. The averaging result of Theorem \ref{APDE} above can be seen as a law of large numbers. The natural next step is then to study the fluctuations of the slow-fast system around its averaged limit, in other words to look for a central limit theorem. For this purpose, we introduce the process
\begin{equation}
z^\e_t=\frac{u^\e_t-u_t}{\sqrt\e}
\end{equation}
for $t\in[0,T]$ and $\e\in]0,1]$.
\begin{theorem}\label{DEDP}
When $\e$ goes to $0$, the process $z^\e$ converges in distribution in $\mathcal{C}([0,T],H)$ towards a process $z$ uniquely defined as the solution of the following martingale problem: for any measurable, bounded and twice Fr\'echet differentiable function $\psi:H\to\R$, the process 
\begin{equation}\label{m:prob}
\bar N_{\psi}(t):=\psi(z_t)-\int_0^t\bar{\mathcal{G}}^1(s)\psi(z_s)ds
\end{equation}
with
\begin{equation}\label{gen1:moy}
\bar{\mathcal{G}}^1(t)\psi(z)
=\frac{d\psi\mathbf{1}}{dz}(z)[\D z+\frac{dF_{\bar r}}{du}(u_t)[z]]+{\rm Tr }\frac{d^2\psi\mathbf{1}}{dz^2}(z)C_{\bar r}(u_t)
\end{equation}
for $t\in[0,T]$, is a martingale. We denote by $\mathbf{1}$ the vector $(1,\cdots,1)^t\in\R^N$. The diffusion operator $C_{\bar r}(u)=B_{\bar r}(u)B_{\bar r}(u)^*:H^*\to H^*$ is characterized by
\begin{eqnarray*}
&~&<C_{\bar r}(u)e^*_j,e_i>\\
&=&\int_{E_{\bar r(1)}\times\cdots E_{\bar r(N)}}<G_r(u)-F_{\bar r}(u),e_i><\Phi(r,u),e_j>\mu_{\bar r(1)}(u)(dr(1))\cdots\mu_{\bar r(N)}(u)(dr(N))
\end{eqnarray*}
where $e^*_j(e_i)=\delta_{ij}$ (Kronecker symbol) for $i,j\geq1$. Moreover $\Phi: H\times \RR\to H^*$ is the unique solution of
\begin{equation}\label{def:Phi:ms}
\left\{\begin{array}{ccc}
\Ba(u)\Phi(r,u)&=&-(G_r(u)-F_{\bar r}(u)),\quad \forall (u,r)\in H\times\RR\\
\int_{E_{j_1}\times\cdots E_{j_N}}\Phi(u,r)\mu_{j_1}(u)(dr)\times\cdots\times\mu_{j_{N}}(u)(dr)&=&0,\quad \forall (j_1,\cdots,j_N)\in \{1,\cdots,l\}^N
\end{array}\right.
\end{equation}
where $\Ba$ is the "fast" generator.
\end{theorem}
\noindent The evolution equation associated to the martingale problem (\ref{m:prob}) is the following hybrid SPDE (see \cite{DPZ})
\begin{equation}\label{eq:dev}
dz_t=\left(\D z_t +\frac{dF_{\bar r_t}}{du}(u_t)[z_t]\right)dt+B_{\bar r_t}(u_t)dW_t
\end{equation}
with initial condition $0$ and zero Dirichlet boundary conditions. $W$ denotes the standard cylindrical Wiener process on the Hilbert space $H$. Formally, the cylindrical Wiener process $W$ is defined as follows: let $((\beta_k(t))_{t\geq0},k\geq1)$ be a family of independent brownian motions, then
\[
W_t=\sum_{k\geq1}\beta_k(t)e_k
\]
See \cite{DPZ} for more information about the construction of $W$. A complete description of the diffusion operator $C$ is provided is Section \ref{Diff_op}. The well definitness of equation (\ref{eq:dev}) may be an issue since the operator $C_j(u)$ is not of trace class in $H$ for $(j,u)\in \{1,\cdots,l\}\times H$. However, for any $(j,u)\in \{1,\cdots,l\}\times \mathcal{C}([0,T],H)$ and $t>0$, the operator
\[
Q^j_t: \psi\mapsto\int_0^te^{\D (t-s)}C_j(u_s)e^{\D (t-s)}\psi ds
\]
is of trace class in $L^2(I)$. Thus we can apply classical results from the theory of SPDE in Hilbert spaces to deduce the existence and uniqueness of a mild solution to equation (\ref{eq:dev}), see the classical reference  \cite{DPZ} on this topic. See also \cite{Yin:Zhu} for an introduction to switching diffusions.\\
A natural step after having obtained a central limit theorem corresponding to an averaged model is to look at the associated Langevin approximation. Formally, the Langevin approximation corresponds to the averaged model \emph{plus} fluctuations. In our case this results in the study of the process $\tilde u^\e$ defined in the all-fast case for $\e>0$ by
\begin{equation}\label{Lang_ini}
d\tilde u^\e=[\Delta \tilde u^\e + F_{\bar r}(\tilde u^\e)]dt+\sqrt{\e}B_{\bar r}(\tilde u^\e)dW_t
\end{equation}
with initial condition $u_0$ and zero Dirichlet boundary conditions. However, it appears that the operator
\[
Q^j_t: \psi\in\mapsto \int_0^t e^{\D (t-s)}C_j(u_s)e^{\D (t-s)}\psi ds
\]
with $(j,u)\in \{1,\cdots,l\}\times \mathcal{C}([0,T],H)$ is not of trace class in $H$ (see the example in Section \ref{Exemple}) while by definition, the reaction term $F$ is well defined from $H$ to $H^*$. This implies technical issues in trying to define properly a solution to equation (\ref{Lang_ini}). In order to circumvent this difficulty, we consider the reaction term $F_{\bar r,\phi}$ instead of $F_{\bar r}$ as follows
\[
F_{\bar r,\phi}(u)=\frac{1}{N}\sum_{i\in \ZZ}\sum_{j=1}^l 1_{j}(\bar r(i))\sum_{\zeta\in E_j}c_\zeta \mu_j((u,\phi_{z_i})_{L^2(I)})(\zeta)(v_\zeta-(u,\phi_{z_i})_{L^2(I)})\phi_{z_i}
\]
where for $i\in\ZZ$ the functions $\phi_{z_i}$ are defined on $I$ by
\[
\phi_{z_i}(x)=\frac1\kappa M\left(\frac{x-z_i}{\kappa}\right)
\]
with $\kappa$ small enough such that $\phi_{z_i}$ is compactly supported in $I$. The mollifier $M$ is defined on $\R$ by
\[
M(x)=e^{-\frac{1}{1-x^2}}1_{[-1,1]}(x)
\]
The corresponding non-averaged reaction term is given by
\[
G_{r,\phi}(u)=\frac{1}{N}\sum_{i\in \ZZ}c_{r(i)}(v_{r(i)}-(u,\phi_{z_i})_{L^2(I)})\phi_{z_i}
\]
for $(r,u)\in\RR\times L^2(I)$. For $i\in\ZZ$, the function $\phi_{z_i}$ is therefore an approximation in the sense of distributions of the Dirac distribution $\delta_{z_i}$. Notice that each $\phi_{z_i}$ is in $L^2(I)$. The biological meaning of replacing $\delta_{z_i}$ by $\phi_{z_i}$ is that when the channel located at $z_i$ is open and allows a current to pass, it is not just the voltage at $z_i$ that is affected, but also the voltage on a small area around $z_i$ (see \cite{BR}). The smaller $\kappa$, the smaller this the area. When $u$ is held fixed, the dynamic of the ion channel at location $z_i$ is evolves according to
\begin{equation}\label{saut:phi}
\PP(r_{t+h}(i)=\zeta|r_t(i)=\xi)=\alpha_{\xi,\zeta}\left((u,\phi_{z_i})_{L^2(I)}\right)h+o(h)
\end{equation}
for $\xi,\zeta\in E$ and $t,h\geq0$. It is easy to see that all the results stated in the present paper for the the model described in Section 1 are still valid for this class of models. In this case the averaged model is given by
\begin{equation}\label{Av:phi}
\partial_tu=[\Delta u + F_{\bar r,\phi}(u)]dt
\end{equation}
with initial condition $u_0$ and zero Dirichlet boundary conditions. The coupled averaged dynamic of the ion channels obeys
\[
\PP(\bar r_{t+h}(i)=k|\bar r_t(i)=j)=\bar\alpha_{jk}((u_t,\phi_{z_i}))_{L^2(I)})h+o(h)
\]
for $i\in\ZZ$, $j,k\in\{1,\cdots,l\}$ and $t,h\geq0$.\\
We would like to compare the above averaged equation with the following Langevin approximation
\begin{equation}\label{Lang:phi}
d\tilde u^\e=[\Delta \tilde u^\e + F_{\bar r,\phi}(\tilde u^\e)]dt+\sqrt{\e}B_{\bar r,\phi}(\tilde u^\e)dW_t
\end{equation}
where $B_{\bar r,\phi}$ is defined as the squared root of $C_{\bar r,\phi}=B_{\bar r,\phi} B^*_{\bar r,\phi}$ and $C_{\bar r,\phi}$ is as in Theorem (\ref{DEDP}), taking into account the changes due to the new form of the model.
\begin{proposition}\label{A1}
The following estimate holds, where the trace is taken in the $L^2(I)$-sense
\[
\text{Tr }\int_0^t e^{\D (t-s)}C_{\bar r,\phi}(u_s)e^{\D (t-s)}ds\leq \sum_{k\geq1}\int_0^t(\alpha\|u_s\|^2_{L^2(I)}+\beta\|u_s\|_{L^2(I)}+\gamma )e^{-2(k\pi)^2(t-s)}ds 
\]
for any $t\in[0,T]$, and any functions $u\in\mathcal{C}([0,T],L^2(I))$ and averaged state $\bar r\in\{1,\cdots,l\}$ with $\alpha,\beta,\gamma$ three constants.
\end{proposition}
\noindent In particular, the operators $Q^j_t$ are of trace class in $L^2(I)$ and the Langevin approximation of $u$ is then well defined.
\begin{proposition}\label{lang:def}
Let $\e>0$. The hybrid SPDE:
\begin{equation}\label{lang}
d\tilde u^\e=[\Delta \tilde u^\e + F_{\bar r,\phi}(\tilde u^\e)]dt+\sqrt{\e}B_{\bar r,\phi}(\tilde u^\e)dW_t
\end{equation}
with initial condition $u_0$ and zero Dirichlet boundary condition, has a unique solution with sample paths in $\mathcal{C}([0,T],L^2(I))$. Moreover 
\begin{equation}\label{bound:lang}
\sup_{t\in[0,T]}\EE(\|\tilde u^\e_t\|^2_{L^2(I)})<+\infty
\end{equation}
\end{proposition}
\noindent We can now compare the Langevin approximation to the averaged model $u$.
\begin{theorem}\label{thm:Lang}
Let $T>0$ held fixed. The Langevin approximation satisfies 
\begin{equation}
\lim_{\e\to0}\EE\left(\sup_{t\in[0,T]}\|u_t-\tilde u^\e_t\|^2_{L^2(I)}\right)=0.
\end{equation}
Therefore, it is indeed an approximation of $u$.
\end{theorem}

\subsection{General Framework}\label{GF}

The arguments developed for the averaging in \cite{GT} as well as those leading to a central limit theorem and a Langevin approximation presented in the present paper are also valid in a more general setting that we describe in the present section. As mentioned in the Introduction, we provide detailed proofs for the stochastic Hodgkin-Huxley model. Their extension to the general class of PDMP described below is straightforward. We first describe a set of assumptions on the infinite dimensional PDMP that will ensure these results. Let $A$ be a self-adjoint linear operator on a Hilbert space $H$ such that there exists a Hilbert basis $\{e_k,k\geq1\}$ of $H$ made up with eigenvectors of $A$
\[
Ae_k=-\l_k e_k
\]
for $k\geq1$ and such that:
\[
\sup_{k\geq1}\sup_{y\in I}|e_k(y)|<\infty
\]
The eigenvalues $\{\l_k,k\geq1\}$ are assumed to form an increasing sequence of positive numbers enjoying the following property
\[
\sum\frac{1}{\l_k}<\infty
\]
Let $\RR$ be a finite space. For any $r\in\RR$, the reaction term $G_r:H\mapsto H$ is globally Lipschitz on $H$ uniformly on $r\in\RR$. That is to say, there exists a constant $L_G>0$ such that for any $(r,u,\tilde u)\in \RR\times H\times H$ we have
\[
\|G_r(u)-G_r(\tilde u)\|_H\leq L_G\|u-\tilde u\|_H
\]
For fixed $u\in H$ let $Q(u):=(q_{r\tilde r}(u))_{(r,\tilde r)\in\RR\times\RR}$ be an intensity matrix on $\RR$ associated to a continuous time Markov chain $(r_t,t\geq0)$. We assume that for $r\neq\tilde r$, the intensity rate functions $q_{r\tilde r}:H\mapsto\R_+$ are uniformly bounded and Lipschitz. There exists two constants $B_q,L_q$ such that for any $(r,\tilde r,u,\tilde u)\in \RR\times \RR\times H\times H$ we have
\[
\sup_{(r,\tilde r)\in\RR\times\RR}\sup_{u\in H}q_{r\tilde r}(u)\leq B_q,\quad |q_{r\tilde r}(u)-q_{r\tilde r}(\tilde u)|\leq L_q\|u-\tilde u\|_H
\]
Moreover the intensity rate functions are supposed to be uniformly bounded below: there exists a positive constant $q_-$such that
\[
\inf_{(r,\tilde r)\in\RR\times\RR}\inf_{u\in H}q_{r\tilde r}(u)\geq q_-
\]
We also assume that there exists a unique pseudo-invariant measure $\mu(u)$ associated to the generator $Q(u)$ which is bounded and Lipschitz in $u$.\\
Let $T>0$ be a fixed time horizon. We consider the following two time-scale evolution problem
\begin{equation}\label{gene_e}
\left\{\begin{array}{ccc}
\partial_t u^\e_t&=&Au^\e_t+G_{r_t}(u^\e_t)\\
\PP(r^\e_{t+h}=\tilde r|r^\e_t=r)&=& \frac1\e q_{r\tilde r}(u^\e_t)h+o(h)
\end{array}
\right.
\end{equation}
for $t\in[0,T]$, with $\e>0$. This is the so called all-fast case: there is only one class of fast transitions in the state space $\RR$ and this is $\RR$ itself. The multi-classes case can be deduced from the all-fast case without difficulty. Let us first state the averaging result.
\begin{theorem}
The process $(u^\e_t,t\in[0,T])$ converges in law in $\mathcal{C}([0,T],H)$ towards $u$ solution of the averaged problem
\[
\partial_t u_t=Au_t+F(u_t)
\]
for $t\in[0,T]$, where the averaged reaction term is given for $u\in H$ by
\[
F(u)=\int_\RR G_r(u)\mu(u)(dr)
\]
\end{theorem}
\noindent Let $z^\e$ be the fluctuation around the averaged limit. For $t\in[0,T]$, $z_t$ is defined as
\[
z^\e_t=\frac{u^\e_t-u_t}{\sqrt{\e}}
\]
Then, the central limit theorem takes the following form.
\begin{theorem}
The process $(z^\e_t,t\in[0,T])$ converges in law in $\mathcal{C}([0,T],H)$ towards $z$ solution of the SPDE
\[
dz_t=[Az_t+\frac{dF}{du}(u_t)[z_t]]dt+B(u_t)dW_t
\]
The diffusion operator $C(u)=B(u)B(u)^*$ is characterized by
\begin{equation*}
<C(u)e^*_j,e_i>
=\int_{\RR}(G_r(u)-F(u),e_i)(\Phi(r,u),e_j)\mu(u)(dr)
\end{equation*}
where $e^*_j(e_i)=\delta_{ij}$ for $i,j\geq1$. Moreover $\Phi$ is the unique solution of
\begin{equation}\label{def:Phi:ms}
\left\{\begin{array}{cccc}
\Ba(u)\Phi(r,u)&=&-(G_r(u)-F(u))& \forall (u,r)\in H\times\RR\\
\int_\R\Phi(r,u)\mu(u)(dr)&=&0&
\end{array}\right.
\end{equation}
\end{theorem}
\noindent We can then consider the Langevin approximation $(\tilde u^\e_t,t\in[0,T])$ solution of the SPDE
\[
d\tilde u^\e_t=[A\tilde u^\e_t+F(\tilde u^\e_t)]dt+\sqrt{\e}B(\tilde u^\e_t)dW_t
\]
for $t\in[0,T]$. The Langevin approximation is well defined in $\mathcal{C}([0,T],H)$ and the following proposition holds.
\begin{proposition}
Let $T>0$. The Langevin approximation $(\tilde u^\e_t,t\in[0,T])$ satisifes
\begin{equation}
\lim_{\e\to0}\EE\left(\sup_{t\in[0,T]}\|u_t-\tilde u^\e_t\|^2_H\right)=0
\end{equation}
Therefore it is indeed an approximation of the averaged process $u$.

\end{proposition}

\section{Proofs}\label{Preuve}

We now proceed to prove Theorem \ref{DEDP} and Theorem \ref{thm:Lang}. In Theorem \ref{DEDP}, we want to prove the convergence in distribution of the process $z^\e$ when $\e$ goes to zero. As usual in this context such a proof can be divided in two parts: the proof of tightness of the family $\{z^\e,\e\in]0,1]\}$ which implies that there exists a convergent subsequence and the identification of the limit which allows us to characterize the limit and prove its uniqueness. We write in full details the proof in the all fast case, that is when all the states in $E$ communicate at fast rates of order $\frac1\e$. In this case there is a unique class of fast communications which is the whole state space $E$ (in this case $l=1$ w.r.t. the notation of section \ref{Sing}). The multi-scale case (when $l>1$) can be deduced from the all fast case and is mainly a complication in the notations. The comments on the extension to the multi-scale case are postponed to Section \ref{Multi}. Section \ref{Langevin} is devoted to the proof of Theorem \ref{thm:Lang}.

\subsection{The all fast case}\label{All}

For the sake of clarity, we first rewrite the statement of Theorem \ref{DEDP} in the all fast case. When all states in $E$ communicate at fast rates, for each $\e>0$, the generator of the process $(u^\e,r^\e)$ is given by 
\begin{equation}\label{gene_eps_fast}
\mathcal{A^\e}f(u^\e_t,r^\e_t)=\frac{d}{dt}f(u^\e_\cdot,r^\e_t)(t)+\frac1\e\Ba (u^\e_t)f(u^\e_t,\cdot)(r^\e_t)
\end{equation}
where the slow part of the generator reduces to zero and
\begin{equation*}
\Ba(u^\e_t) f(u^\e_t,r^\e_t)=\sum_{i\in\ZZ}\sum_{\xi\in E}[f(u^\e_t,r^\e_t(r^\e_t(i)\to\zeta))-f(u^\e_t,r^\e_t)]\alpha_{r^\e_t(i),\xi}(u^\e_t(z_i))
\end{equation*}
where $r_t(r_t(i)\to\zeta)$ is the element of $\RR$ with $r_t(r_t(i)\to\zeta)(j)$ equal to $r_t(j)$ if $j\neq i$ and to $\zeta$ if $j=i$. For any $i\in\ZZ$ with $u\in H$ held fixed, the Markov process $r(i)$ has a unique stationary distribution $\mu(u(z_i))$. Then the process $(r(i),i\in\ZZ)$ has the following stationary distribution
\begin{equation*}
\mu(u)=\bigotimes_{i\in\ZZ} \mu(u(z_i))
\end{equation*}
The averaged generalized function reduces to
\begin{eqnarray}\label{F}
F(u)&=&\int_{\RR}G_r(u)\mu(u)(dr)\\
&=&\frac{1}{N}\sum_{\xi\in E}\sum_{i\in \ZZ}c_\xi \mu(u(z_i))(\xi)(v_\xi-u(z_i))\delta_{z_i}\nonumber
\end{eqnarray}
The averaged limit $u$ is solution of the PDE
\[
\partial_tu_t=\D u_t+F(u_t)
\]
with initial condition $u_0$ and zero Dirichlet boundary conditions. In this case Theorem \ref{DEDP} reads as follows.
\begin{theorem}\label{DEDP_fast}
When $\e$ goes to $0$  the process $z^\e$ converges in distribution in $\mathcal{C}([0,T],H)$ towards a process $z$. This process is uniquely determined as the solution of the following martingale problem: for any measurable, bounded and twice Fr\'echet differentiable function $\psi: H\to\R$, the process 
\begin{equation}\label{m:prob:fast}
\bar N_{\psi}(t):=\psi(z_t)-\int_0^t\bar{\mathcal{G}}^1(s)\psi(z_s)ds
\end{equation}
where
\begin{eqnarray*}\label{gen1:moy}
&~&\bar{\mathcal{G}}^1(t)\psi(z)\\
&=&\frac{d\psi\mathbf{1}}{dz}(z)[\D z+\frac{dF}{du}(u_t)[z]]+{\rm Tr }\frac{d^2\psi\mathbf{1}}{dz^2}(z)C(u_t)
\end{eqnarray*}
for $t\in[0,T]$, is a martingale. We denote by $\mathbf{1}$ the vector $(1,\cdots,1)^t\in\R^N$. The diffusion operator $C(u)=B(u)B(u)^*$ is characterized by
\begin{eqnarray*}
&~&<C(u)e^*_j,e_i>\\
&=&\int_{\RR}<G_r(u)-F(u),e_i><\Phi(r,u),e_j>\mu(u)(dr)
\end{eqnarray*}
where $e^*_j(e_i)=\delta_{ij}$ for $i,j\geq1$. $\Phi$ is the unique solution of the equation
\begin{equation}\label{def:Phi:ms}
\left\{\begin{array}{ccc}
\Ba(u)\Phi(r,u)&=&-(G_r(u)-F(u))\\
\int_\R\Phi(r,u)\mu(u)(dr)&=&0
\end{array}\right.
\end{equation}
\end{theorem}
\noindent The associated evolution equation is the following stochastic partial differential equation
\begin{equation}\label{SPDE_fast}
dz_t=\left(\D z_t +\frac{dF}{du}(u_t)[z_t]\right)dt+B(u_t)dW_t
\end{equation}
with initial condition $0$ and zero Dirichlet boundary conditions. We remark that in this case, the limiting process $z$ is solution of an SPDE which is no longer \emph{hybrid}.

\subsection{Tightness}\label{Tight}

To show that the family $\{z^\e,\e\in]0,1]\}$ is tight in $\mathbb{D}([0,T],H)$, we use Aldous criterion (cf. \cite{metivier}) which can be splitted in two parts as follows.

\begin{criterion}[General criterion for tightness \cite{metivier}]\label{GCT}
Let us assume that the family $\{z^\e,\e\in]0,1]\}$ satisfies Aldous's condition: for any $\delta,M>0$, there exist $\eta,\e_0>0$ such that for all stopping times $\tau$ with $\tau+\eta<T$,
\begin{equation}\label{Ald}
\sup_{\e\in]0,\e_0]}\sup_{\theta\in]0,\eta[}\PP(\|z^\e_{\tau+\theta}-z^\e_\tau\|_H\geq M)\leq\delta
\end{equation}
and moreover, for each $t\in[0,T]$, the family $\{z^\e_t,\e\in]0,1]\}$ is tight in $H$. Then $\{z^\e,\e\in]0,1]\}$ is tight in $\mathbb{D}([0,T],H)$.
\end{criterion}
\begin{criterion}[Tightness in a Hilbert space \cite{metivier}]\label{TH}
Let $H$ be a separable Hilbert space endowed with a basis $\{e_k,k\geq1\}$ and for $k\geq1$ define
\[
L_k=span\{e_i,1\leq i\leq k\}
\]
Then  $(z^\e_t,\e\in]0,1])$ is tight in $H$ if, and only if, for any $\delta,\eta>0$ there exist $\rho,\e_0>0$ and $L_{\delta,\eta}\subset \{L_k,k\geq1\}$ such that

\begin{equation}\label{TH1} 
\sup_{\e\in]0,\e_0]} \PP(\|z^\e_t\|_H>\rho)\leq\delta
\end{equation}
\begin{equation}\label{TH2} \sup_{\e\in]0,\e_0]} \PP(d(z^\e_t,L_{\delta,\eta})>\eta)\leq\delta
\end{equation}
where $d(z^\e_t,L_{\delta,\eta})=\inf_{v\in L_{\delta,\eta}}\|z^\e_t-v\|_H$ is the distance of $z^\e$ to the subspace $L_{\delta,\eta}$.
\end{criterion}
\noindent We begin by showing that for a fixed $t\in[0,T]$, the family $\{z^\e_t,\e\in]0,1]\}$ is uniformly bounded in $L^2(\Omega,H^1_0(I))$.
\begin{proposition}\label{bound:z}
There exists a constant $C$ depending only on $T$ but otherwise neither on $t\in[0,T]$ nor on $\e\in]0,1]$ such that
\[
\EE(\|z^\e_t\|^2_H)\leq C
\]
In particular, for any fixed $t\in[0,T]$, condition (\ref{TH1}) is satisfied by the family $\{z^\e_t,\e\in]0,1]\}$. 
\end{proposition}
\begin{proof}
Let $t\in[0,T]$ and $\e\in]0,1]$ be fixed. We have:
\begin{eqnarray*}
&~&\frac{d}{dt}\|u^\e_t-u_t\|^2_H\\
&=&2<\partial_t(u^\e_t-u_t),u^\e_t-u_t>\\
&=&2<\D(u^\e_t-u_t),u^\e_t-u_t>+2<G_{r^\e_t}(u^\e_t)-F(u_t),u^\e_t-u_t>\\
&=&2<\D(u^\e_t-u_t),u^\e_t-u_t>+2<G_{r^\e_t}(u^\e_t)-F(u^\e_t),u^\e_t-u_t>\\
&+&2<F(u^\e_t)-F(u_t),u^\e_t-u_t>
\end{eqnarray*}
almost surely. We treat each of the above terms separately. From the spectral decomposition of the Laplacian operator we deduce that the first term satisfies
\[
2<\D(u^\e_t-u_t),u^\e_t-u_t>\leq-2\pi^2\|u^\e_t-u_t\|^2_H.
\]
Regarding the third term, we notice that the application $u\mapsto <F(u),u>$ is locally Lipschitz on $H$ and that the quantities $u^\e_t$ and $u_t$ are uniformly bounded w.r.t. $t\in[0,T]$ and $\e\in]0,1]$ thanks to Propositions \ref{APDE} and \ref{Prop_bound}. Thus there exists a constant $C$, independent of $t\in[0,T]$ and $\e\in]0,1]$, such that
\[
2<F(u^\e_t)-F(u_t),u^\e_t-u_t>\leq C\|u^\e_t-u_t\|^2_H
\]
Integrating over $[0,t]$ and taking expectation yields the following inequality
\begin{eqnarray*}
&~&\EE(\|u^\e_t-u_t\|^2_H)\\
&\leq& \EE(\|u^\e_0-u_0\|^2_H)+2(C-\pi^2)\int_0^t\EE(\|u^\e_s-u_s\|^2_H)ds\\
&+&\EE\left(\int_0^t2<G_{r^\e_s}(u^\e_s)-F(u^\e_s),u^\e_s-u_s>ds\right)
\end{eqnarray*}
Let us consider the latter of these terms. Using the same approach as the one developed for the identification of the limit in the proof of the averaging result in \cite{GT}, we deduce the existence of a constant $C(T)$ depending only on $T$ such that
\[
\left|\EE\left(\int_0^t2<G_{r^\e_s}(u^\e_s)-F(u^\e_s),u^\e_s-u_s>ds\right)\right|\leq C(T)\e
\]
For the sake of completeness, we review now briefly this approach and refer to \cite{GT} for more details. The first point is to show, as in Proposition 6 of \cite{GT}, that there exists a measurable and bounded function $f:H\times\RR\times[0,T]\to\R$ such that $\int_\mathcal{R} f(u,r,t)\mu(u)(dr)=0$ and for all $(u,r,t)\in H\times\RR\times[0,T]$
\begin{equation}\label{Poisson}
\mathcal{B}(u)f(u,\cdot,t)(r)=<G_{r}(u)-F(u),u-u_t>
\end{equation}
Then using the regularity of the mappings $(u,r,t)\in H\times\RR\times[0,T]\mapsto<G_{r}(u)-F(u),u-u_t>$ and the operator $\Ba(u)$ for $u\in H$, we deduce that the application $(u,r,t)\in H\times\RR\times[0,T]\mapsto<G_{r}(u)-F(u),u-u_t>$ is bounded, Fr\'echet differentiable in $u$ with bounded Fr\'echet derivative and differentiable in $t$ with bounded derivative. Using the general theory of Markov processes, we deduce that there exists a martingale $M^\e$ such that
\begin{eqnarray*}
&~&f(u^\e_t,r^\e_t,t)\\
&=&f(u^\e_0,r^\e_0,0)+\int_0^t\mathcal{A}^\e f(u^\e_s,r^\e_s,s)ds+M^\e_t\\
&=&f(u^\e_0,r^\e_0,0)+\frac1\e\int_0^t\mathcal{B}(u^\e_s)f(u^\e_s,r^\e_s,s)ds+\int_0^t\frac{d}{ds}f(u^\e_\cdot,r^\e_s,s)(s)+\frac{d}{ds}f(u^\e_s,r^\e_s,\cdot)(s)ds+M^\e_t\\
&=&f(u^\e_0,r^\e_0,0)+\frac1\e\int_0^t<G_{r^\e_s}(u^\e_s)-F(u^\e_s),u^\e_s-u_s>ds\\
&+&\int_0^t\frac{d}{ds}f(u^\e_\cdot,r^\e_s,s)(s)+\frac{d}{ds}f(u^\e_s,r^\e_s,\cdot)(s)ds+M^\e_t
\end{eqnarray*}
Therefore
\begin{eqnarray*}
&~&\int_0^t<G_{r^\e_s}(u^\e_s)-F(u^\e_s),u^\e_s-u_s>ds\\
&=&\e f(u^\e_t,r^\e_t,t)-\e f(u^\e_0,r^\e_0,0)-\e\int_0^t\frac{d}{ds}f(u^\e_\cdot,r^\e_s,s)(s)-\frac{d}{ds}f(u^\e_s,r^\e_s,\cdot)(s)ds-\e M^\e_t
\end{eqnarray*}
Taking the expectation , using the fact that $M^\e$ is a martingale and that $f$ is regular, we obtain the desired estimate.\\
Assembling all the above estimates we obtain
\begin{eqnarray*}
&~&\EE(\|u^\e_t-u_t\|^2_H)\\
&\leq& \EE(\|u^\e_0-u_0\|^2_H)+ C(T)\e+2(C-\pi^2)\int_0^t\EE(\|u^\e_s-u_s\|^2_H)ds
\end{eqnarray*}
Since  $u^\e_0=u_0$ a standard application of the Gronwall's lemma yields the desired result. We fend up this proof by showing that for any fixed $t\in[0,T]$, the family $\{z^\e_t,\e\in]0,1]\}$ fulfills the requirement (\ref{TH1}). Indeed, let $\delta>0$ and denote by $C$ the constant independent of $\e$ and $t\in[0,T]$ such that
\[
\EE(\|z^\e_t\|^2_H)\leq C
\]
By the Markov inequality we have, for $\rho>0$,
\[
\sup_{\e\in]0,1]}\PP(\|z^\e_t\|_H>\rho)\leq \sup_{\e\in]0,1]}\frac{\EE(\|z^\e_t\|^2_H)}{\rho^2}\leq \frac{C}{\rho^2}
\]  
and for $\rho$ large enough, we obtain that $\sup_{\e\in]0,1]}\PP(\|z^\e_t\|_H>\rho)<\delta$.
\end{proof}
\noindent Notice that by the compact embedding of $H$ in $L^2(I)$, it follows from Proposition \ref{bound:z} that for any fixed $t\in[0,T]$ the family $\{z^\e_t,\e\in]0,1]\}$ is tight in $L^2(I)$. However, this is not sufficient for our purpose and we now prove the tightness of the family $\{z^\e_t,\e\in]0,1]\}$ in $H$ for any fixed $t\in[0,T]$. This is the object of the following propositions.

\begin{proposition}\label{proche:z}
Let $t\in]0,T]$ and for $p\geq1$ let us define the following truncation
\[
z^{\e,p}_t=\sum_{k=1}^p(z^\e_t,e_k)e_k
\]
Then
\[
\lim_{p\to\infty}\EE{\|z^\e_t-z^{\e,p}_t\|^2_H}=0
\]
uniformly in $\e\in]0,1]$.
\end{proposition}
\begin{proof}
For a fixed $k\geq1$ we have
\begin{eqnarray*}
\frac{d}{dt}(z^\e_t,e_k)^2&=&2(z^\e_t,e_k)\frac{d}{dt}(z^\e_t,e_k)\\
&=&2(z^\e_t,e_k)(-(k\pi)^2\left(z^\e_t,e_k)+\frac{1}{\sqrt\e}<G_{r^\e_t}(u^\e_t)-F(u_t),e_k>\right)\\
&=&-2(k\pi)^2(z^\e_t,e_k)^2+\frac{2}{\sqrt\e}(z^\e_t,e_k)<F(u^\e_t)-F(u_t),e_k>\\
&+&\frac{2}{\sqrt\e}(z^\e_t,e_k)<G_{r^\e_t}(u^\e_t)-F(u^\e_t),e_k>
\end{eqnarray*}
almost surely. We recall the interpretation $<\partial_{z_i},e_k>=(1+(k\pi)^2)e_k(z_i)$. Then  a direct computation using the arguments developed in the proof of Proposition \ref{bound:z}, yields the existence of a constant $C(T)$ independent of $\e\in]0,1]$ such that
\[
(z^\e_t,e_k)^2\leq C(T)(1+(k\pi)^2)-(k\pi)^2\int_0^t (z^\e_s,e_k)^2ds
\]
almost surely. Using Gronwall's lemma we deduce that
\[
(z^\e_t,e_k)^2\leq C(T)(1+(k\pi)^2)e^{-(k\pi)^2t}
\]
The result follows since the series $\sum (1+(k\pi)^2)e^{-(k\pi)^2t}$ is convergent for $t>0$.
\end{proof}

\noindent We now check that the family $\{z^\e,\e\in]0,1]\}$ satisfies the first part of Aldous criterion that we have called Criterion \ref{GCT}. 

\begin{proposition}\label{aldous:z}
Let $\tau>0$ be a stopping time and $\theta>0$ such that $\tau+\theta\leq T$. There exists a constant $C$ depending only on $T$ such that
\[
\EE(\|z^\e_{\tau+\theta}-z^\e_\tau\|^2_H)\leq C\theta
\] 
\end{proposition}
\begin{proof}
We notice that for $t>0$ and $\theta>0$ such that $t+\theta\leq T$
\[
\partial_\theta(z^\e_{t+\theta}-z^\e_t)=\D z^\e_{t+\theta}+\frac{1}{\sqrt\e}(G_{r^\e_{t+\theta}}(u^\e_{t+\theta})-F(u_{t+\theta})).
\]
Thus almost surely
\begin{eqnarray*}
\frac{d}{d\theta}\|z^\e_{t+\theta}-z^\e_t\|^2_H&=&2<\D z^\e_{t+\theta},z^\e_{t+\theta}-z^\e_t>\\
&+&\frac{2}{\e}<G_{r^\e_{t+\theta}}(u^\e_{t+\theta})-F(u_{t+\theta}),u^\e_{t+\theta}+u_t-u_{t+\theta}-u^\e_t>
\end{eqnarray*}
The first term satisfies
\begin{eqnarray*}
<\D z^\e_{t+\theta},z^\e_{t+\theta}-z^\e_t>&\leq&\| z^\e_{t+\theta}\|_H\|z^\e_{t+\theta}-z^\e_t\|_H\\
&\leq&\frac12\| z^\e_{t+\theta}\|^2_H+\frac12\|z^\e_{t+\theta}-z^\e_t\|^2_H
\end{eqnarray*}
which is uniformly bounded in $t,\theta$ and $\e$ in expectation by Proposition \ref{bound:z}. For the second term, the arguments developed in the proof of Proposition \ref{bound:z} yield the existence of a constant $C$ depending only of $T$ such that
\[
\EE(\int_0^\theta<G_{r^\e_{t+s}}(u^\e_{t+s})-F(u^\e_{t+s}),u^\e_{t+s}+u_t-u_{t+s}-u^\e_t> ds)\leq C\theta\e
\]
and thus, still denoting by $C$ a constant depending only of $T$
\[
\EE(\|z^\e_{t+\theta}-z^\e_t\|^2_H)\leq C\theta
\]
The same arguments apply when replacing $t$ by the stopping time $\tau$.
\end{proof}
\noindent According to Criteria \ref{GCT} and \ref{TH}, Propositions \ref{bound:z}, \ref{proche:z} and  \ref{aldous:z}, the family $\{z^\e,\e\in]0,1]\}$ is tight in $\mathbb{D}([0,T],H)$. The continuity of each of the elements of the family implies that $\{z^\e,\e\in]0,1]\}$ is tight in $\mathcal{C}([0,T],H)$.

\subsection{Identification of the limit}\label{Id}

For an introduction to the approach used in this section in a similar, but finite dimensional, case, we refer to \cite{WTh}. We want to prove uniqueness of an accumulation piont $z$ of $\{z^\e,\e\in]0,1]\}$ and identify the limit. For this purpose, we study the process $(z^\e,r^\e)$ for $\e\in]0,1]$. Notice first that the process $z^\e$ satisfy the following equation
\begin{eqnarray*}
\partial_t z^\e_t&=&\D z^\e_t+\frac{1}{\sqrt\e}(G_{r^\e_t}(u^\e_t)-F(u_t))\\
&=&\D z^\e_t+\frac{1}{\sqrt\e}(G_{r^\e_t}(u_t+\sqrt\e z^\e_t)-F(u_t))
\end{eqnarray*}
where we have used that: $u^\e_t=u_t+u^\e_t-u_t=u_t+\sqrt\e z^\e_t$. The initial condition for $z^\e$ is $0$ and the boundary conditions are still zero Dirichlet boundary conditions.\\
Let $\phi:H\times\RR\times\R_+$ be a real valued, measurable and bounded function of class $\mathcal{C}^2$ on $H$ and $\mathcal{C}^1$ on $\R_+$. We write down the generator of the process $(z^\e,r^\e)$ against $\phi$. For $(z,r,t)\in H\times\RR\times\R_+$ we have
\begin{eqnarray}\label{gen}
&~&\mathcal{G}(t)\phi(z,r,t)\\
&=&\frac{d\phi}{dz}(z,r,t)[\D z+\frac{1}{\sqrt\e}(G_{r}(u_t+\sqrt\e z)-F(u_t))]\\
&+&\frac1\e\Ba(u_t+\sqrt\e z)\phi(z,r,t)+\partial_t\phi(z,r,t)
\end{eqnarray}
Following the usual theory of Markov processes, the process $(M^\e_\phi(t),t\in[0,T])$ defined for $t\geq0$ by
\[
M^\e_\phi(t)=\phi(z^\e_t,r^\e_t,t)-\int_0^t\mathcal{G}(s)\phi(z^\e_s,r^\e_s,s)ds
\]
is a martingale for the natural filtration associated to the process $(z^\e,r^\e)$. We want to identify the terms of different orders in $\e$ of this martingale process. For this purpose, we choose a function $\phi$ with the following decomposition
\[
\phi(z,r,t)=\psi(z,r)+\sqrt\e\gamma(z,r,t)
\]
where the functions $\psi$ and $\gamma$ have the same regularity as $\phi$. We write the Taylor expansion in $\e$ of the two following terms
\begin{eqnarray*}
G_r(u_t+\sqrt\e z)&=&G_r(u_t)+\sqrt\e\frac{dG_r}{du}(u_t)[z]+O(\e)\\
\Ba(u_t+\sqrt\e z)&=&\Ba(u_t)+\sqrt\e\frac{d\Ba}{du}(u_t)[z]+O(\e)
\end{eqnarray*}
Inserting this expansion in the expression of the generator (\ref{gen}) we want the terms of order $\frac1\e$ to vanish and this yields for  $(z,r,t)\in H\times\RR\times\R_+$
\begin{equation}
\Ba(u_t)\psi(z,r)=0
\end{equation}
That is to say, the application $\psi$ does not depend of $r\in\RR$ and is of the form
\[
\psi(z,r)=\psi(z)\mathbf{1}
\]
where $\psi:H\to\R$ is of class $\mathcal{C}^2$ and $\mathbf{1}$ is the vector $(1,\cdots,1)^t\in\R^N$.
The generator is then of the following form, where we gather the terms of the same order in $\e$
\begin{eqnarray*}
&~&\mathcal{G}(t)\phi(z,r,t)\\
&=&\frac{1}{\sqrt\e}\left(\frac{d\psi\mathbf{1}}{dz}(z)[G_r(u_t)-F(u_t)]+\Ba(u_t)\gamma(z,r,t)+\frac{d\Ba}{du}(u_t)[z]\psi(z)\mathbf{1}\right)\\
&+&\frac{d\psi\mathbf{1}}{dz}(z)[\D z+\frac{dG_r}{du}(u_t)[z]]+\frac{d\gamma}{dz}(z,r,t)[G_r(u_t)-F(u_t)]+\frac{d\Ba}{du}(u_t)[z]\gamma(z,r,t)\\
&+&\sqrt\e\left(\partial_t\gamma(z,r,t)+\frac{d\gamma}{dz}(z,r,t)[\D z+\frac{dG_r}{du}(u_t)[z]]\right)\\
&+&O(\sqrt{\e})
\end{eqnarray*}
We now want the terms of order $\frac{1}{\sqrt\e}$ to vanish, that is to say, for $(z,r,t)\in H\times\RR\times\R_+$
\[
\frac{d\psi\mathbf{1}}{dz}(z)[G_r(u_t)-F(u_t)]+\Ba(u_t)\gamma(z,r,t)+\frac{d\Ba}{du}(u_t)[z]\psi(z)\mathbf{1}=0
\]
Notice that $\Ba(u_t)\mathbf{1}=0$ implies that for all $(z,t)\in H\times\R_+$
\[
\frac{d\Ba}{du}(u_t)[z]\psi(z)\mathbf{1}=0
\]
and we are left with the equation
\begin{equation}\label{rac:eps}
\Ba(u_t)\gamma(z,r,t)=-\frac{d\psi\mathbf{1}}{dz}(z)[G_r(u_t)-F(u_t)]
\end{equation}
We look for $\gamma$ of the form:
\[
\gamma(z,r,t)=\frac{d\psi\mathbf{1}}{dz}(z)[\Phi(r,u_t)]
\]
where $\Phi:H\times\RR\to H$ has to be identified. Inserting the above expression of $\gamma$ in (\ref{rac:eps}) we obtain
\[
\frac{d\psi\mathbf{1}}{dz}(z)[\Ba(u_t)\Phi(r,u_t)]=-\frac{d\psi\mathbf{1}}{dz}(z)[G_r(u_t)-F(u_t)]
\]
And thus it is enough that for any $(u,r)\in H\times\RR$
\begin{equation}\label{fred:z}
\Ba(u)\Phi(r,u)=-(G_r(u)-F(u))
\end{equation}
To ensure uniqueness of the solution for equation (\ref{fred:z}) we impose moreover the condition
\[
\int_\RR \Phi(r,u)\mu(u)(dr)=0
\]
Then, from the definition of $F$ we have $\int_\RR \left(G_r(u)-F(u)\right)\mu(u)(dr)=0$ and the Fredholm alternative ensures us of the existence and uniqueness of $\Phi$ solution of equation (\ref{fred:z}).\\
We have identify the terms of order $1$ in $\e$ of the generator of the process $(z^\e,r^\e)$. We are left to show that the terms of order $1$ in $\e$ correspond, after averaging, to the generator of the process $z$. For $(z,r,t)\in H\times\RR\times\R_+$ we define
\begin{eqnarray*}\label{gen1}
&~&\mathcal{G}^1(t,r)\psi(z)\\
&=&\frac{d\psi\mathbf{1}}{dz}(z)[\D z+\frac{dG_r}{du}(u_t)[z]]+\frac{d^2\psi\mathbf{1}}{dz^2}(z)[\Phi(r,u_t),G_r(u_t)-F(u_t)]\\
&+&\frac{d\Ba}{du}(u_t)[z]\frac{d\psi\mathbf{1}}{dz}(z)[\Phi(r,u_t)]
\end{eqnarray*}
Let us define also the following process
\[
N^\e_{\psi}(t)=\psi(z^\e_t)-\int_0^t\mathcal{G}^1(s,r^\e_s)\psi(z^\e_s)ds
\]
By construction we see that $\EE(|M^\e_\phi(t)-N^\e_{\psi}(t)|^2)=O(\e)$. When $\e$ goes to $0$, by the averaging result we see that the term $\int_0^t\mathcal{G}^1(s,r^\e_s)\psi(z^\e_s)ds$ should converge to
\[
\int_0^t\int_\RR\mathcal{G}^1(s,r)\psi(z_s)\mu(u_s)(dr)ds
\]
Therefore, the aim of the end of the proof is then to show that the process defined by
\[
\bar N_{\psi}(t)=\psi(z_t)-\int_0^t\bar{\mathcal{G}}^1(s)\psi(z_s)ds
\]
where $z$ is an accumulation point of the family $\{z^\e,\e\in]0,1]\}$ and
\begin{eqnarray*}\label{gen1:moy}
&~&\bar{\mathcal{G}}^1(t)\psi(z)\\
&=&\frac{d\psi\mathbf{1}}{dz}(z)[\D z+\frac{dF}{du}(u_t)[z]]+\frac{d^2\psi\mathbf{1}}{dz^2}(z)\int_\RR[\Phi(r,u_t),G_r(u_t)-F(u_t)]\mu(u_t)(dr)
\end{eqnarray*}
is a martingale for the natural filtration associated to the process $(z_t,t\geq0)$. It is not straightforward because we have no information on the asymptotic behavior of the process $(z^\e,r^\e)$ when $\e$ goes to $0$.
\begin{proof}
Let $0\leq t_1\leq t_2\leq\cdots\leq t_k\leq s\leq t$ be $k+2$ reals, with $k\geq1$ an integer. For $i\in\{1,\cdots,k\}$, we take a measurable and bounded function $g_i$. In order to show that the process $(\bar N_{\psi}(t),t\geq0)$ is a martingale for the natural filtration associated to the process $(z_t,t\geq0)$ we will prove that
\[
\EE((\bar N_{\psi}(t)-\bar N_{\psi}(s))g_1(z_{t_1})\cdots g_k(z_{t_k}))=0
\]
We write, using elementary substitution and the fact that $z^\e$ converges in law toward $z$ when $\e$ goes to $0$
\begin{eqnarray*}
&~&\EE((\bar N_{\psi}(t)-\bar N_{\psi}(s))g_1(z_{t_1})\cdots g_k(z_{t_k}))\\
&=&\EE((\psi(z_t)-\psi(z_s)-\int_s^t\bar{\mathcal{G}}^1(l)\psi(z_l)dl)g_1(z_{t_1})\cdots g_k(z_{t_k}))\\
&=&\EE((\psi(z_t)-\psi(z_s))g_1(z_{t_1})\cdots g_k(z_{t_k}))-\EE((\int_s^t\bar{\mathcal{G}}^1(l)\psi(z_l)dl)g_1(z_{t_1})\cdots g_k(z_{t_k}))\\
&=&\lim_{\e\to0}\EE((\psi(z^\e_t)-\psi(z^\e_s))g_1(z^\e_{t_1})\cdots g_k(z^\e_{t_k}))-\EE((\int_s^t\bar{\mathcal{G}}^1(l)\psi(z_l)dl)g_1(z_{t_1})\cdots g_k(z_{t_k}))\\
&=&\lim_{\e\to0}\EE((N^\e_{\psi}(t)-N^\e_{\psi}(s))g_1(z^\e_{t_1})\cdots g_k(z^\e_{t_k}))\\
&+&\lim_{\e\to0}\EE((\int_s^t\mathcal{G}^1(l,r^\e_l)\psi(z^\e_l)dl)g_1(z^\e_{t_1})\cdots g_k(z^\e_{t_k}))-\EE((\int_s^t\bar{\mathcal{G}}^1(l)\psi(z_l)dl)g_1(z_{t_1})\cdots g_k(z_{t_k}))
\end{eqnarray*}
On one hand, by the definition of $N^\e_\psi$ and the previous study of the different order in $\e$ of the martingale $M^\e_\phi$ we see that
\begin{eqnarray*}
&~&\lim_{\e\to0}\EE((N^\e_{\psi}(t)-N^\e_{\psi}(s))g_1(z^\e_{t_1})\cdots g_k(z^\e_{t_k}))\\
&=&\lim_{\e\to0}\EE((N^\e_{\psi}(t)-N^\e_{\psi}(s))g_1(z^\e_{t_1})\cdots g_k(z^\e_{t_k}))-\EE((M^\e_{\phi}(t)-M^\e_{\phi}(s))g_1(z^\e_{t_1})\cdots g_k(z^\e_{t_k}))\\
&=&\lim_{\e\to0} O(\sqrt\e)\\
&=&0
\end{eqnarray*}
On the other hand, for $\e^1>0$ which will be chosen later
\begin{eqnarray*}
&~&\lim_{\e\to0}\EE((\int_s^t\mathcal{G}^1(l,r^\e_l)\psi(z^\e_l)dl)g_1(z^\e_{t_1})\cdots g_k(z^\e_{t_k}))-\EE((\int_s^t\bar{\mathcal{G}}^1(l)\psi(z_l)dl)g_1(z_{t_1})\cdots g_k(z_{t_k}))\\
&=&\lim_{\e\to0}\EE((\int_s^t\mathcal{G}^1(l,r^{\e}_l)\psi(z^\e_l)dl)g_1(z^\e_{t_1})\cdots g_k(z^\e_{t_k}))-\EE((\int_s^t\mathcal{G}^1(l,r^{\e^1}_l)\psi(z^\e_l)dl)g_1(z^\e_{t_1})\cdots g_k(z^\e_{t_k}))\\
&+&\lim_{\e\to0}\EE((\int_s^t\mathcal{G}^1(l,r^{\e^1}_l)\psi(z^\e_l)dl)g_1(z^\e_{t_1})\cdots g_k(z^\e_{t_k}))-\EE((\int_s^t\mathcal{G}^1(l,r^{\e^1}_l)\psi(z_l)dl)g_1(z_{t_1})\cdots g_k(z_{t_k}))\\
&+&\lim_{\e\to0}\EE((\int_s^t\mathcal{G}^1(l,r^{\e^1}_l)\psi(z_l)dl)g_1(z_{t_1})\cdots g_k(z_{t_k}))-\EE((\int_s^t\bar{\mathcal{G}}^1(l)\psi(z_l)dl)g_1(z_{t_1})\cdots g_k(z_{t_k}))\\
\end{eqnarray*}
We know that the quantity
\[
\EE((\int_s^t\mathcal{G}^1(l,r^{\e^1}_l)\psi(z_l)dl)g_1(z_{t_1})\cdots g_k(z_{t_k}))-\EE((\int_s^t\bar{\mathcal{G}}^1(l)\psi(z_l)dl)g_1(z_{t_1})\cdots g_k(z_{t_k}))
\]
can be made arbitrarily small by conditioning appropriately (as in the proof of Proposition \ref{bound:z} for example) by choosing $\e^1$ small enough.  Then, since $z^\e$ converges in law towards $z$ when $\e$ goes to $0$, we have
\begin{eqnarray*}
&\lim_{\e\to0}&\EE((\int_s^t\mathcal{G}^1(l,r^{\e^1}_l)\psi(z^\e_l)dl)g_1(z^\e_{t_1})\cdots g_k(z^\e_{t_k}))-\EE((\int_s^t\mathcal{G}^1(l,r^{\e^1}_l)\psi(z_l)dl)g_1(z_{t_1})\cdots g_k(z_{t_k}))\\
&=&0
\end{eqnarray*}
This shows that
\[
\lim_{\e\to0}\EE((\int_s^t\mathcal{G}^1(l,r^\e_l)\psi(z^\e_l)dl)g_1(z^\e_{t_1})\cdots g_k(z^\e_{t_k}))-\EE((\int_s^t\bar{\mathcal{G}}^1(l)\psi(z_l)dl)g_1(z_{t_1})\cdots g_k(z_{t_k}))=0
\]
and finally
\[
\EE((\bar N_{\psi}(t)-\bar N_{\psi}(s))g_1(z_{t_1})\cdots g_k(z_{t_k}))=0
\]
as announced.
\end{proof}
\noindent Thus the limit process $z$ is solution of the following martingale problem: for any measurable, bounded and twice Fr\'echet differentiable function $\psi$ the process defined for $t\in[0,T]$ by
\[
\bar N_{\psi}(t)=\psi(z_t)-\int_0^t\bar{\mathcal{G}}^1(s)\psi(z_s)ds
\]
where:
\begin{eqnarray*}\label{gen1:moy}
&~&\bar{\mathcal{G}}^1(t)\psi(z)\\
&=&\frac{d\psi\mathbf{1}}{dz}(z)[\D z+\frac{dF}{du}(u_t)[z]]+\frac{d^2\psi\mathbf{1}}{dz^2}(z)\int_\RR[\Phi(r,u_t),G_r(u_t)-F(u_t)]\mu(u_t)(dr)
\end{eqnarray*}
is a martingale. \\
The limit process $z$ is therefore solution of the martingale problem associated to the operator $\bar{\mathcal{G}}^1$. Then $z$ is a solution of the SPDE (\ref{DEDP_fast}) where the diffusion operator $C(u)$ for $u\in H$ is identified thanks to the relation
\[
\frac{d^2\psi\mathbf{1}}{dz^2}(z)\int_\RR[\Phi(r,u),G_r(u)-F(u)]\mu(u)(dr)=\text{Tr }\frac{d^2\psi\mathbf{1}}{dz^2}(z)C(u)
\]
for $(u,z)\in H\times H$. The uniqueness of $z$ follows from the properties of the Laplacian operator, the reaction term $\frac{dF}{du}$ and the trace class operator $C(u)$. For more insight in the properties of the diffusion operator, see the following section. 

\subsection{On the diffusion operator $C$}\label{Diff_op}

In this section, we give more precision on the diffusion operator $C$. First, we explicit  the function $\Phi$ in function of the data of our problem.

\begin{proposition}[First representation of the diffusion operator]
For $u\in H$ and $r\in\R$ we have:
\[
\Phi(r,u)=-(\mu(u)^*\mu(u)+\Ba^*(u)\Ba(u))^{-1}\Ba^*(u)(G_\cdot(u)-F(u))(r)
\]
That is, the function $\Phi(\cdot,u)$ is explicitly given in function of the intensity operator $\Ba(u)$ and the associated invariant measure $\mu(u)$.
\end{proposition}

\begin{proof}
The application $\Phi$ is defined by the two conditions
\begin{equation}\label{def:Phi}
\left\{\begin{array}{ccc}
\Ba(u)\Phi(r,u)&=&-(G_r(u)-F(u))\\
\int_\RR\Phi(r,u)\mu(u)(dr)&=&0
\end{array}\right.
\end{equation}
for $(u,r)\in H\times\RR$. Let $u\in H$ held fixed. Defining $D(u)=(\mu(u),\Ba(u))^T$, we have that (\ref{def:Phi}) reduces to
\[
D(u)\Phi(\cdot,u)=-\left(\begin{array}{c}
0\\G_\cdot(u)-F(u)
\end{array}\right)
\]
Then
\[
D^*(u)D(u)\Phi(\cdot,u)=-D^*(u)\left(\begin{array}{c}
0\\G_\cdot(u)-F(u)
\end{array}\right)
\]
The key point is that the operator $D^*(u)D(u)$ is invertible. Indeed
\[
D^*(u)D(u)=\mu(u)^*\mu(u)+\Ba^*(u)\Ba(u)
\]
and the kernel of the two operators $\mu(u)^*\mu(u)$ and $\Ba^*(u)\Ba(u)$ are in direct sum and span the whole space  $\R^{N}$. Let ${\rm x}\in \R^{|\ZZ|}$, then ${\rm x}={\rm z}+{\rm y}$ with ${\rm z}\in {\rm Ker }\mu(u)^*\mu(u)$ and ${\rm y}\in {\rm Ker }\Ba^*(u)\Ba(u)$. We have
\[
\mu(u)^*\mu(u){\rm y}+\Ba^*(u)\Ba(u){\rm z}=0
\]
Since $\Ba(u)\mathbf{1}=0$ (where $\mathbf{1}\in \R^{|\RR|}$) and $\mu(u)\mathbf{1}=1$, multiplying the above equation to the left by $\mathbf{1}^T$ we have
\[
\mu(u){\rm y}=0
\]
Since $y\in {\rm Ker }\Ba^*(u)\Ba(u)={\rm Ker }\Ba(u)={\rm span}\mathbf{1}$ here, we have ${\rm y}=y\mathbf{1}$ with $y\in\R$ and
\[
\mu(u){\rm y}=\mu(u)y\mathbf{1}=y
\]
and thus $y=0$ and ${\rm y}=0$. Therefore ${\rm x}={\rm z}\in{\rm Ker }\mu(u)^*\mu(u)$ and $\Ba^*(u)\Ba(u){\rm z}=0$. Thus $z\in {\rm Ker }\mu(u)^*\mu(u)\cap {\rm Ker }\Ba^*(u)\Ba(u)=\{0\}$ and ${\rm x}={\rm z}=0$. The operator $D^*(u)D(u)$ is then invertible.
\end{proof}

\begin{proposition}[Second representation of the diffusion operator]\label{SRDO}
For any $(u,r)\in H\times\RR$
\[
\Phi(u,r)=\int_0^\infty \EE_r(G_{r^u_s}(u)-F(u))ds
\]
where $r^u$ has the same law of $r$ when $u$ is held fixed.
\end{proposition}
\begin{proof}
It is a direct consequence of the fact that the process defined by
\[
M_t=\Phi(u,r^u_t)-\Phi(u,r)-\int_0^t\Ba(u)\Phi(u,r^u_s)ds
\]
is a martingale for the natural filtration associated to the process $r^u$. Then, taking expectation and recalling that
\begin{equation}\label{def:Phi2}
\left\{\begin{array}{ccc}
\Ba(u)\Phi(r,u)&=&-(G_r(u)-F(u))\\
\int_\RR\Phi(r,u)\mu(u)(dr)&=&0
\end{array}\right.
\end{equation}
yields:
\[
\EE_r(\Phi(u,r^u_t))=\Phi(u,r)-\int_0^t\EE_r(G_{r^u_s}(u)-F(u))ds
\]
Since:
\[
\lim_{t\to\infty}\EE_r(\Phi(u,r^u_t))=\int_\RR\Phi(r,u)\mu(u)(dr)=0
\]
the desired result follows.
\end{proof}

\begin{proposition}
The diffusion operator $C(u)$, for $u\in H$, is positive in the sense that
\[
\text{Tr } C(u)\geq0
\]
Therefore the operator $B(u)$ is well defined.
\end{proposition}
\begin{proof}
For $u\in H$ we have:
\begin{eqnarray*}
\text{Tr }C(u)&=&\sum_{k\geq1}\int_\RR<G_r(u)-F(u_t),e_k><\Phi(r,u),e_k>\mu(u)(dr)\\
&=&-\sum_{k\geq1}\int_\RR<\Ba(u)\Phi(r,u),e_k><\Phi(r,u),e_k>\mu(u)(dr)\\
&\geq&0
\end{eqnarray*}
because of the negativity of the operator $\Ba(u)$: all the eigenvalues of the operator $\Ba(u)$ are non positive.
\end{proof}

\subsection{The multi-scale case}\label{Multi}

Let us recall that by multi-scale case we mean the case when there is at least two spaces $E_i$ or in other words, the ionic channels exhibit different time scales.\\
For the tightness and the identification of the limit, this is, as said before, mainly a complication in the notations and this is why we will be very brief in this section. We refer for example, in the finite dimensional case, to the work of \cite{WTh} or in our framework for the proof of averaging, to Section 3.2 of \cite{GT}. For another instructive example dealing with slow-fast continuous Markov chain, see \cite{YZ}.

\subsection{Langevin approximation}\label{Langevin}

This section is devoted to the study of the Stochastic Hodgkin-Huxley model with mollifier. The model is presented in Section \ref{Main}. More particularly, we are interested in this section by the Langevin approximation of the averaged model. We begin by the proof of Proposition \ref {A1} and \ref{lang:def}. As for the Central Limit Theorem we detail the proof only in the all-fast case.

\begin{proposition}
The following estimate holds, where the trace is taken in the $L^2(I)$-sense,
\[
\text{Tr }\int_0^t e^{\D (t-s)}C_\phi(u_s)e^{\D (t-s)}ds\leq \sum_{k\geq1}\int_0^t(\alpha\|u_s\|^2_{L^2(I)}+\beta\|u_s\|_{L^2(I)}+\gamma )e^{-2(k\pi)^2(t-s)}ds 
\]
for all $t\in[0,T]$ and all functions $u\in\mathcal{C}([0,T],L^2(I))$ with $\alpha,\beta,\gamma$ three constants.
\end{proposition}
\begin{proof}
This is a direct consequence of Proposition \ref{SRDO}. For an explicit example, see Section \ref{Exemple}.
\end{proof}
\noindent In particular, the operator $Q_t$ defined by
\[
Q_t: \psi\in\mapsto \int_0^t e^{\D (t-s)}C(u_s)e^{\D (t-s)}\psi ds
\]
with $(j,u)\in \{1,\cdots,l\}\times \mathcal{C}([0,T],L^2(I))$ is of trace class in $L^2(I)$. The Langevin approximation of $u$ is then well defined as stated in the following proposition. We recall that in the all-fast case
\[
 F_\phi(u)=\frac{1}{N}\sum_{i\in \ZZ}\sum_{\xi\in E}c_\xi \mu((u,\phi_{z_i})_{L^2(I)})(\xi)(v_\xi-(u,\phi_{z_i})_{L^2(I)})\phi_{z_i}
\]
for $u\in L^2(I)$.
\begin{proposition}
Let $\e>0$. The SPDE
\begin{equation}
d\tilde u^\e=[\Delta \tilde u^\e + F_\phi(\tilde u^\e)]dt+\sqrt{\e}B_\phi(\tilde u^\e)dW_t
\end{equation}
with initial condition $u_0$ and zero Dirichlet boundary condition has a unique solution with sample paths in $\mathcal{C}([0,T],L^2(I))$. Moreover the quantity
\begin{equation}\label{bound:lang}
\sup_{t\in[0,T]}\EE(\|\tilde u^\e_t\|^2_{L^2(I)})
\end{equation}
is finite.
\end{proposition}
\begin{proof}
Thanks to the properties of the laplacian operator, the local Lipschitz continuity of $F_\phi$ and Proposition \ref{A1}, we can apply classical results on SPDE, see for example \cite{DPZ}. The bound (\ref{bound:lang}) is obtained using arguments similar to those used in the proof of the following theorem.
\end{proof}
\noindent We proceed to the proof of Theorem \ref{thm:Lang}.
\begin{theorem}
Let $T>0$ held fixed. The Langevin approximation is indeed an approximation of $u$
\begin{equation}
\lim_{\e\to0}\EE\left(\sup_{t\in[0,T]}\|u_t-\tilde u^\e_t\|^2_{L^2(I)}\right)=0
\end{equation}
\end{theorem}

\begin{proof}
First we notice that we have
\begin{eqnarray*}
&~&\tilde u^\e_t-u_t\\
&=&\int_0^t e^{\D (t-s)}(F_\phi(\tilde u^\e_s)-F_\phi(u_s))ds+\sqrt\e\int_0^t e^{\D (t-s)}B_\phi(\tilde u^\e_s)dW_s
\end{eqnarray*}
We deal with the two above terms separately. Recall that for any $u\in L^2(I)$
\begin{eqnarray*}
F_\phi(u)&=&\frac{1}{N}\sum_{\xi\in E}\sum_{i\in \ZZ}c_\xi \mu((u,\phi_{z_i})_{L^2(I)})(\xi)(v_\xi-(u,\phi_{z_i})_{L^2(I)})\phi_{z_i}\\
&=&\frac{1}{N}\sum_{i\in \ZZ}f((u,\phi_{z_i})_{L^2(I)})\phi_{z_i}
\end{eqnarray*}
If moreover $\tilde u\in {L^2(I)}$ we have
\[
|f((\tilde u,\phi_{z_i})_{L^2(I)})-f((u,\phi_{z_i})_{L^2(I)})|\leq\max_{\xi\in E} c_\xi(1+\max_{\xi\in E}|v_\xi|+\|u\|_{L^2(I)})\|\tilde u-u\|_{L^2(I)}
\]
We have, using the fact that the quantity $\sup_{t\in[0,T]}\|u_t\|_{L^2(I)}$ is finite almost-surely
\begin{eqnarray*}
&~&\left\|\int_0^t e^{\D (t-s)}(F_\phi(\tilde u^\e_s)-F_\phi(u_s))ds\right\|^2_{L^2(I)}\\
&=&\left\|\sum_{k\geq1} \frac{1}{N}\sum_{i\in \ZZ}\int_0^t(f((\tilde u^\e_s,\phi_{z_i})_{L^2(I)})-f((u_s,\phi_{z_i})_{L^2(I)}))e^{-(k\pi)^2 (t-s)}ds (\phi_{z_i},f_k)_{L^2(I)}f_k\right\|^2_{L^2(I)}\\
&=&\sum_{k\geq1}\left( \frac{1}{N}\sum_{i\in \ZZ}\int_0^t (f((\tilde u^\e_s,\phi_{z_i})_{L^2(I)})-f((u_s,\phi_{z_i})_{L^2(I)}))e^{-(k\pi)^2(t-s)}ds (\phi_{z_i},f_k)_{L^2(I)}\right)^2\\
&\leq&\frac{C_1}{N}\sum_{i\in \ZZ}\sum_{k\geq1}(\phi_{z_i},f_k)^2_{L^2(I)}\left( \int_0^t (f((\tilde u^\e_s,\phi_{z_i})_{L^2(I)})-f((u_s,\phi_{z_i})_{L^2(I)}))e^{-(k\pi)^2(t-s)}ds \right)^2\\
&\leq&\frac{C_1}{N}\sum_{i\in \ZZ}\sum_{k\geq1}(\phi_{z_i},f_k)^2_{L^2(I)}\left( \int_0^t \max_{\xi\in E} c_\xi(1+\max_{\xi\in E}|v_\xi|+\|u_s\|_{L^2(I)})\|\tilde u^\e-u_s\|_{L^2(I)}ds \right)^2\\
&\leq&\frac{C_2}{N}\sum_{i\in \ZZ}\sum_{k\geq1}(\phi_{z_i},f_k)^2_{L^2(I)}\left(\int_0^t\|\tilde u^\e-u_s\|_{L^2(I)}ds \right)^2
\end{eqnarray*}
where $C_1$ and $C_2$ are two deterministic constants depending only on $T$. Thus, there exists a constant $C_3$ depending only on $T$ such that
\[
\left\|\int_0^t e^{\D (t-s)}(F_\phi(\tilde u^\e_s)-F_\phi(u_s))ds\right\|_{L^2(I)}\leq C_3\int_0^t\|\tilde u^\e-u_s\|_{L^2(I)}ds
\]
We prove now that there exists a constant $C_4$ such that
\begin{equation}\label{eqn2}
\EE\left(\sup_{t\in[0,T]}\left\|\int_0^te^{\D(t-s)}B(\tilde u^\e_s)dW_s\right\|^2_{L^2(I)}\right)\leq C_4
\end{equation}
Using the Burkholder-Davis-Gundy inequality we see that
\begin{eqnarray*}
&~&\EE\left(\sup_{t\in[0,T]}\left\|\int_0^te^{\D(t-s)}B_\phi(\tilde u^\e_s)dW_s\right\|^2_{L^2(I)}\right)\\
&=&\EE\left(\sup_{t\in[0,T]}\sum_{k\geq1}\left(\int_0^t<e^{\D(t-s)}B_\phi(\tilde u^\e_s)f_k,f_k>d\beta_k(s)\right)^2\right)\\
&\leq&\sum_{k\geq1}\EE\left(\sup_{t\in[0,T]}\left(\int_0^t<e^{\D(t-s)}B_\phi(\tilde u^\e_s)f_k,f_k>d\beta_k(s)\right)^2\right)\\
&\leq&\sum_{k\geq1}\EE\left(\int_0^T<e^{\D(t-s)}B_\phi(\tilde u^\e_s)f_k,f_k>^2ds\right)\\
&\leq&\EE\left(\text{Tr }\int_0^Te^{\D s}C_\phi(\tilde u^\e_{t-s})e^{\D s}ds\right)\\
&\leq&C_4
\end{eqnarray*}
with $C_4$ a constant depending only of $T$ according to Proposition \ref{A1} and Proposition \ref{lang:def}. From the above inequality we obtain that
\begin{eqnarray*}
&~&\EE(\sup_{t\in[0,T]}\|\tilde u^\e_t-u_t\|^2_{L^2(I)})\\
&=&\EE\left(\sup_{t\in[0,T]}\left\|\int_0^t e^{\D (t-s)}(F_\phi(\tilde u^\e_s)-F_\phi(u_s))ds+\sqrt\e\int_0^t e^{\D (t-s)}B(\tilde u^\e_s)dW_s\right\|^2_{L^2(I)}\right)\\
&\leq&2\EE\left(\sup_{t\in[0,T]}\left\|\int_0^t e^{\D (t-s)}(F_\phi(\tilde u^\e_s)-F_\phi(u_s))ds\right\|^2_{L^2(I)}\right)+2\e\EE\left(\sup_{t\in[0,T]}\left\|\int_0^t e^{\D (t-s)}B(\tilde u^\e_s)dW_s\right\|^2_{L^2(I)}\right)\\
&\leq&2\EE\left(\left(C_3\int_0^T\sup_{t\in[0,s]}\|\tilde u^\e_t-u_t\|_{L^2(I)}ds\right)^2\right)+2\e C_4\\
&\leq&2C^2_3 T\int_0^T\EE\left(\sup_{t\in[0,s]}\|\tilde u^\e_t-u_t\|^2_{L^2(I)}\right)ds+2\e C_4
\end{eqnarray*}
A standard application of the Gronwall's lemma yields the result.
\end{proof}

\section{Example}\label{Exemple}

We consider in this section a spatially extended stochastic Morris-Lecar model. Since the seminal work \cite{Morris}, the deterministic Morris-Lecar model is considered as one of the classical mathematical models for investigating neuronal behavior. At first, this model was design to describe the voltage dynamic of the barnacle giant muscle fiber (see \cite{Morris} for a complete description of the deterministic Morris-Lecar model). To take into account the intrinsic variability of the ion channels dynamic, a stochastic interpretation of this class of models has been introduced (see \cite{BR,WTh} and references therein) in which ion channels are modeled by jump Markov processes. The model then falls into the class of stochastic Hodgkin-Huxley models considered in the present paper. Let us proceed to the mathematical description of the spatially extended stochastic Morris-Lecar model. In this model, the total current is of the form
\begin{eqnarray*}
&~&G_{r^\text{K},r^\text{Ca}}(u)\\
&=&\frac{1}{C}\left[\frac{1}{N_\text{K}}\sum_{i\in\ZZ_\text{K}}1_{1}(r^\text{K}(i))c_{\text{K}}(v_{\text{K}}-u(z_i))\delta_{z_i}+\frac{1}{N_\text{Ca}}\sum_{i\in\ZZ_\text{Ca}}1_{1}(r^\text{Ca}(i))c_{\text{Ca}}(v_{\text{Ca}}-u(z_i))\delta_{z_i}+I\right]
\end{eqnarray*}
and the evolution equation for the transmembrane potential 
\[
\partial_t u_t=\frac{a}{2RC}\D u_t+G_{r^\text{K}_t,r^\text{Ca}_t}(u_t)
\]
on $[0,T]\times[0,1]$ and with zero Dirichlet boundary condition. The total current is the sum of the potassium K current,  the calcium Ca current and the impulse $I$. The positive constants $a,R,C$ are relative to the bio-physical properties of the membrane. When the voltage is held fixed, for any $i\in\ZZ_\text{q}$ where $q$ is equal to K or Ca, $r^q$ is a continuous time Markov chain with only two states $0$ for {\it closed} and $1$ for {\it open}. The jump rate from $1$ to $0$ is given by $\beta_q(u(z_i))$ and from $0$ to $1$ by $\alpha_q(u(z_i))$. All the jump rates are bounded below and above by positive constants. We will assume that the potassium ion channels communicate at fast rates of order $\frac1\e$ for a small $\e>0$. The calcium rates are of order $1$. The invariant measure associated to each channel $i\in\ZZ_\text{K}$ is given by
\[
\mu^\text{K}_i(u(z_i))=\left(\frac{\beta_\text{K}(u(z_i))}{\alpha_\text{K}(u(z_i))+\beta_\text{K}(u(z_i))},\frac{\alpha_\text{K}(u(z_i))}{\alpha_\text{K}(u(z_i))+\beta_\text{K}(u(z_i))}\right).
\]
Therefore the averaged applied current is 
\begin{eqnarray*}
F_{r^\text{Ca}}(u)&=&\frac{1}{C}\left[\frac{1}{N_\text{K}}\sum_{i\in\ZZ_\text{K}}\frac{\alpha_\text{K}(u(z_i))}{\alpha_\text{K}(u(z_i))+\beta_\text{K}(u(z_i))}c_{\text{K}}(v_{\text{K}}-u(z_i))\delta_{z_i}\right.\\
&+&\left.\frac{1}{N_\text{Ca}}\sum_{i\in\ZZ_\text{Ca}}1_{1}(r^\text{Ca}(i))c_{\text{Ca}}(v_{\text{Ca}}-u(z_i))\delta_{z_i}+I\right]
\end{eqnarray*}
In this case the application $\Phi$ of Theorem \ref{DEDP} reads as follows, for $(u,r)\in H\times \RR_\text{K}$,
\[
\Phi(u,r)=\frac{1}{C}\frac{1}{N_\text{K}}\sum_{i\in\ZZ_\text{K}}c_{\text{K}}(v_{\text{K}}-u(z_i))\delta_{z_i}\int_0^\infty \EE_r\left(1_{1}(r^{\text{K},u}_s(i))-\frac{\alpha_\text{K}(u(z_i))}{\alpha_\text{K}(u(z_i))+\beta_\text{K}(u(z_i))}\right)ds
\]
Of course, in this case, the law of $(r^{\text{K},u}_s(i),s\geq0)$ can be fully explicited. After some algebra one obtains
\[
\Phi(u,r)=\frac{1}{C}\frac{1}{N_\text{K}}\sum_{i\in\ZZ_\text{K}}\frac{c_{\text{K}}}{\alpha_\text{K}(u(z_i))+\beta_\text{K}(u(z_i))}(v_{\text{K}}-u(z_i))\left(1_1(r(i))-\frac{\alpha_\text{K}(u(z_i))}{\alpha_\text{K}(u(z_i))+\beta_\text{K}(u(z_i))}\right)\delta_{z_i}
\]
Then the diffusion operator is given for $u\in H$ by
\[
<C^\text{K}(u)\phi,\psi>=\frac{1}{N^2_\text{K}}\sum_{i\in\ZZ_\text{K}}c^2_{\text{K}}(v_{\text{K}}-u(z_i))^2\frac{a_\text{K}(u(z_i))b_\text{K}(u(z_i))}{(\alpha_\text{K}(u(z_i))+\beta_\text{K}(u(z_i)))^3}<\delta_{z_i},\phi><\delta_{z_i},\psi>
\]
for $(\phi,\psi)\in H\times H$. From the above expression, we see that for any $u\in H$, $C^\text{K}$ is not of trace class. However, let us consider, for $t\in[0,T]$
\[
Q_t:\phi\in L^2(I)\mapsto \int_0^te^{\D(t-s)}C^\text{K}(u_s)e^{\D(t-s)}\phi ds
\]
where $(u_s,s\in[0,T])$ is the averaged limit. In the $L^2(I)$-sense, we have, for any $t>0$
\begin{eqnarray*}
&~&\text{Tr }Q_t\\
&=&\sum_{k\geq1}\frac{1}{N^2}\sum_{i\in\ZZ_\text{K}}c^2_{\text{K}}\int_0^te^{-2(k\pi)^2(t-s)}(v_{\text{K}}-u_s(z_i))^2\frac{a_\text{K}(u_s(z_i))b_\text{K}(u_s(z_i))}{(\alpha_\text{K}(u_s(z_i))+\beta_\text{K}(u_s(z_i)))^3}2\sin^2(k\pi z_i)ds\\
&\leq&c^2_{\text{K}}(|v_{\text{K}}|+\sup_{t\in[0,T]}\|u_s\|_H)^2\sup_{y\in\R}\frac{a_\text{K}(y)b_\text{K}(y)}{(\alpha_\text{K}(y)+\beta_\text{K}(y))^3}\sum_{k\geq1}\frac{2}{(k\pi)^2}
\end{eqnarray*}
which is a finite quantity.\\
 We present in Figure \ref{Fig1} numerical simulations of the slow fast Morris-Lecar model with no Calcium current for various $\e>0$. The averaged model (denoted by $\e=0$) and the trace of the diffusion operator are also plotted. We set the calcium current equals to zero in our simulations to emphasize the convergence of the slow-fast spatially extended Morris-Lecar model towards the associated averaged model. See \cite{Morris} Figure 2 for simulations of the deterministic finite dimensional Morris-Lecar system with no calcium current. We observe in Figure \ref{Fig1} that averaging affects the model in several ways. As $\e$ goes to zero, the averaged  number of spikes on a fixed time duration increases until finally form a front wave in the averaged model ($\e=0$). In the same time the intensity of the spikes decreases. Let us also mention the fact that the trace of the diffusion operator is higher in the neighborhood of a spike in accordance to $\cite{WTh}$ where the same phenomenon has been observed for the finite dimensional stochastic Morris-Lecar model.
\begin{figure}
\begin{center}
\begin{tabular}{cc}
\includegraphics[width=7cm]{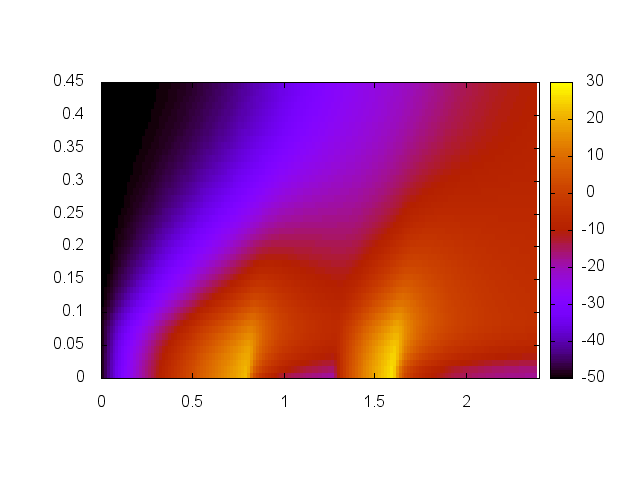}&\includegraphics[width=7cm]{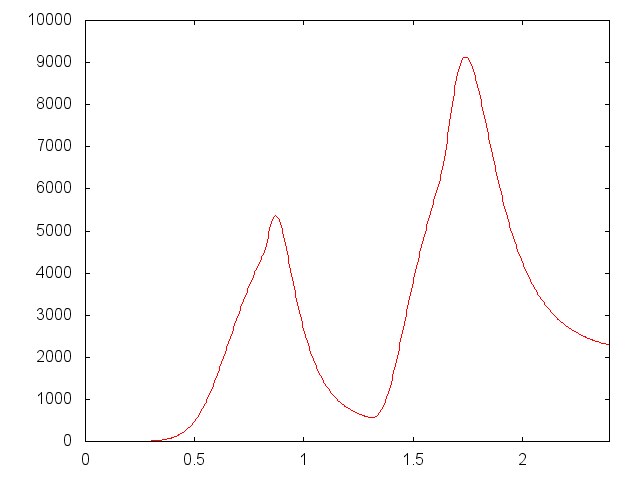}\\
\includegraphics[width=7cm]{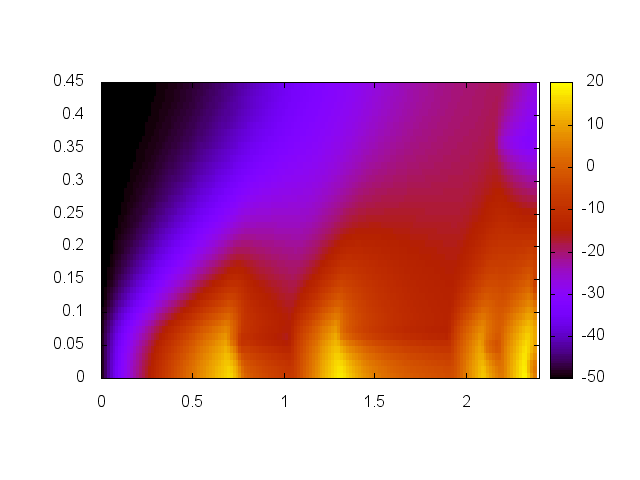}&\includegraphics[width=7cm]{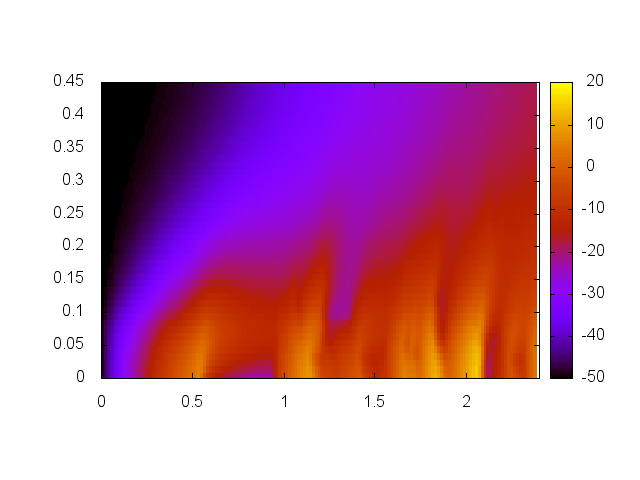}\\
\includegraphics[width=7cm]{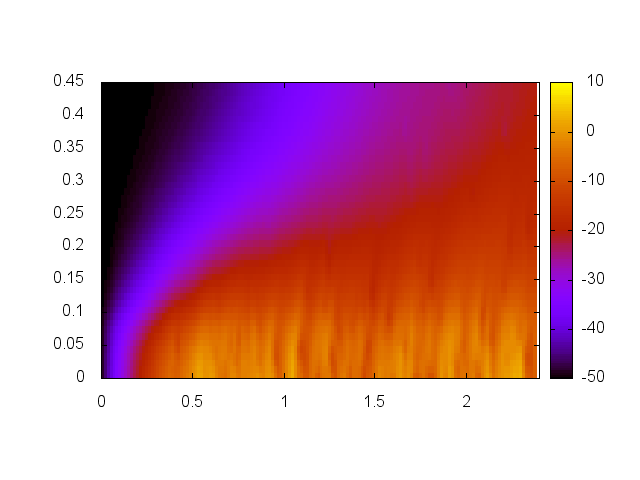}&\includegraphics[width=7cm]{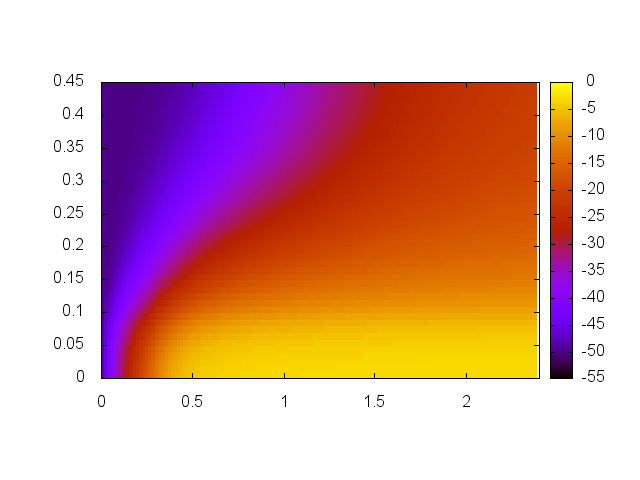}
\end{tabular}
\caption{Simulation of the spatially extended Morris-Lecar model with no Calcium current for $\e$ equal successively to $\e=1$ (up left) $\e=0.1,0.01,0.001,0.0001$ (upper left corner) and $\e=0$ (lower right corner), that is for the averaged model. The plotted curve (up right) corresponds to the corresponding simulation of the Morris Lecar model on its left side, it is the plot of the function $t\mapsto \text{Tr }Q_t$. A stimulus have been injected to the membrane during all the time duration of the simulation on the portion $[0,0.1]$ of the fiber.}\label{Fig1}
\end{center}
\end{figure}

\appendix

\section{Numerical data for the simulations}

Here are the numerical data used for the simulations of the Morris Lecar model
\begin{equation*}
\begin{array}{ccc}
C=1& c_\text{K}=32& v_\text{K}=-70\\
a=1& c_\text{Ca}=0& v_\text{Ca}=0\\
R=0.5& N_\text{K}=50& N_\text{Ca}=0
\end{array}
\end{equation*}
The length of the fiber is $l=0.5$ and the time duration is $T=2.4$. The impulse $I$ is of the form
\[
I(x,t)=\lambda1_{[0,0.1]}(x)
\]
with $\lambda=300$. The data for the internal resistance $R$ and the capacitance $C$ are arbitraly chosen for the purpose of the simulations. The values for the other parameters correspond to \cite{Morris}.

\end{document}